\documentclass{amsart}  	
\usepackage{geometry}                		
\usepackage{graphicx}				
\usepackage{tikz}
\usetikzlibrary{calc}

\usepackage{amssymb}
\usepackage{stmaryrd}
\usepackage{amsthm}

\usepackage{mathtools}
\usepackage{amsfonts}
\usepackage{epsfig}
\usepackage{comment}
\usepackage{color}
\usepackage{amsmath}
\usepackage{dsfont}

\newcommand{\R}{{\mathbb R}}

\newcommand{\N}{{\mathbb N}}

\newcommand{\bM}{{\mathbb M}}

\newcommand{\cD}{{\mathcal D}}

\usepackage{float}

\usepackage{import}
\usepackage{xifthen}
\usepackage{pdfpages}
\usepackage{transparent}

\usepackage{xparse}

\DeclareMathOperator{\diam}{diam}

\DeclareMathOperator{\dist}{dist}

\def\vint_#1{\mathchoice
	{\mathop{\vrule width 6pt height 3 pt depth -2.5pt
			\kern -8pt \intop}\nolimits_{#1}}%
	{\mathop{\vrule width 5pt height 3 pt depth -2.6pt
			\kern -6pt \intop}\nolimits_{#1}}%
	{\mathop{\vrule width 5pt height 3 pt depth -2.6pt
			\kern -6pt \intop}\nolimits_{#1}}%
	{\mathop{\vrule width 5pt height 3 pt depth -2.6pt
			\kern -6pt \intop}\nolimits_{#1}}}

\theoremstyle{plain}
\newtheorem{theorem}{Theorem}[section]

\newtheorem{lemma}[theorem]{Lemma}

\theoremstyle{definition}
\newtheorem{definition}[theorem]{Definition}

\newtheorem{remark}[theorem]{Remark}

\author{Sylvester Eriksson-Bique}
\email{sylvester.d.eriksson-bique@jyu.fi}
\address{Department of Mathematics and Statistics
	P.O. Box 35
	FI-40014 University of Jyväskylä}

\author{Niilo Joutsenlahti}
\email{niilo.joutsenlahti@oulu.fi}
\address{Mathematical Sciences, P.O.Box 8000, FI-90014, University of Oulu, Finland}

\thanks{ S.E.-B. was supported by the Research Council of Finland grant 354241. The N.J was supported by the Jenny and Antti Wihuri Foundation. We thank Guy C. David and Hrant Hakobyan for posing and discussing many of the embedding problems that are solved in this paper.
}

\title[Embeddings of slit carpets]{Quantitative non-embeddability theorems and metric embeddings of slit carpets}

\begin{document}

	\maketitle
	
	\begin{abstract}
		We study the bi-Lipschitz embedding problem for a class of metric spaces called slit carpets. First we show that the $n$th stage $\mathbb{M}_n$ of the standard slit carpet of Merenkov admits a bi-Lipschitz embedding into Euclidean space with distortion $ O(\sqrt{n})$. Then, we show a nearly sharp lower bound of $\Omega\left(\sqrt{\frac{n}{\log(n)}}\right)$.  This result quantifies the recent result by David and Eriksson-Bique, and thus gives a quantified answer to the question 8 in the paper by Heinonen and Semmes by showing that $\mathbb{M}_\infty$ does not bi-Lipschitz embed into Euclidean spaces. Then, we study the $L^1$ embeddability of the standard slit carpet. We show that the standard slit carpet has Lipschitz dimension $1$ in the sense of Cheeger and Kleiner, and consequently prove that it admits a bi-Lipschitz embedding into $ L^1 $.
		
		Third, we generalize the results in terms of targets and domains. First, we give a qualitative and Lebesgue differentiation based argument which shows that general slit carpets do not bi-Lipschitz embed into any Banach space with the RNP property. As a consequence, $\mathbb{M}_\infty$ does not bi-Lipschitz embed to $\ell_1$.  We then consider carpets $\mathbb{M}^a$ where the relative sizes of slits decrease according to a sequence $a\in c_0$. We give a quantitative $\beta$-number based argument which shows that the carpets $\mathbb{M}^a$ do not bi-Lipschitz embed into Hilbert space if $a\not\in \ell_{1+\epsilon}$. 
	\end{abstract}
	
	\section{Introduction}

\noindent In 1998, Heinonen and Semmes asked whether any Ahlfors regular metric space that admits a regular mapping to a Euclidean space also admits a bi-Lipschitz embedding into some Euclidean space \cite[Question~8]{MR1452413}. This problem was solved by David and the first author by studying slit carpets, a class of fractal-like metric spaces, originally introduced by Merenkov \cite{merenkov, seb2020regular}. They showed that despite admitting regular maps, slit carpets do not admit bi-Lipschitz embeddings into Euclidean spaces. Earlier, Merenkov had shown that slit carpets have interesting properties with respect to quasisymmetries: they do not embed into the plane and they have the so-called co-Hopfian property \cite{merenkov}. These carpets were further studied from the perspective of quasisymmetric uniformization in \cite{hakobyan2023quasisymmetric}.  This study raised a number of further questions about the embeddability of these slit carpets: What is the quantitative growth rate of distortion for approximations of the slit carpet? Do there exist embeddings to other targets (e.g. $L^1, \ell_1$)? What happens to embeddability if the carpet is perturbed by changing the sizes or positions of slits? In this work, we give answers to all three of these problems.

We will give precise definitions below, but for this introductory discussion, the slit carpet $\mathbb{M}_\infty$ can be thought of as the space obtained with this process: take the collection of squares consisting of the unit square $[0,1]^2$ and all the dyadic squares $[l2^{-n},(l+1)2^{-n}]\times  [k2^{-n},(k+1)2^{-n}] \subset [0,1]^2, n,l,k \in \mathbb{N}$ inside the unit square. Then for each square $[l2^{-n},(l+1)2^{-n}]\times  [k2^{-n},(k+1)2^{-n}]$ in this collection, remove the vertical segment $\{l2^{-n} + 2^{-n-1}\} \times [k2^{-n} + 2^{-n-2}, k2^{-n} + 3\cdot 2^{-n-2}]$, which is in the center of the square and whose length is half the side length of the square. For this remaining set, assign the shortest path metric and complete the space with respect to this metric. We denote this \textit{slit carpet} space with the symbol $\mathbb{M}_\infty$ and for each fixed $n$, its approximations by the symbol $\mathbb{M}_n$, where slits are only removed at dyadic scales up to $n$. We will also consider two generalizations of this space for which the locations or lengths of the removed slits can vary. But the following results stated in this introduction hold also for this standard version, so we postpone the definitions of the generalized slit carpets. We follow the nomenclature of \cite{hakobyan2023quasisymmetric} and refer to $\mathbb{M}_\infty$ as the \textit{(standard) slit carpet}. 

This paper can be divided into four main parts, which are briefly introduced below.

\subsection{Quantification of the Euclidean bi-Lipschitz distortion of $\mathbb{M}_\infty$}
In the first part of this paper, we use the techniques from \cite{seb2020regular} to give us a quantitative lower bound on the bi-Lipschitz distortion of $\mathbb{M}_n$.  Quantification here means that we study how fast the distortion grows as a function of $n$ for the finite approximations $\mathbb{M}_n$. Let $C_n$ be the smallest constant s.t. there is a $(1,C_n)$-bi-Lipschitz embedding $f:\mathbb{M}_n\to \R^N$ for some $N$. We study how $C_n$ grows in terms of $n$.

For instance, for the well-known Laakso graph (see Figure~\ref{fig:laakso}) we know that the space is not bi-Lipschitz embeddable to any Hilbert space and for the bi-Lipschitz distortion $C_n$ of the finite approximations of the Laakso graph we know $C_n \in \Theta(\sqrt{n})$ \cite{laakso2000ahlfors, lang2001bilipschitz}. Although the slit carpet has many properties similar to the Laakso graph, the techniques applicable to the Laakso graph fail when applied to the slit carpet. The main intuitive reason is that in the Laakso graph the slits on the next level are inside the previous level, but in the slit carpet the slits are not nested.

The Assouad embedding theorem implies that for the bi-Lipschitz distortion $C_n$ of $\mathbb{M}_n$ we have $C_n \in O(\sqrt{n})$ \cite{assouad1983plongements}. (See section \ref{nota} for the precise definitions). Interestingly, we also find a nearly sharp lower bound for the distortion,
\begin{theorem}\label{theorem1}
	$ C_n \in \Omega\left(\sqrt{\frac{n}{\log(n)}}\right) $ and $ C_n \in O\left(\sqrt{n}\right)$, where $ C_n$ denotes the optimal Euclidean distortion of $ \mathbb{M}_n $. 
\end{theorem}
In Section \ref{sec:upperbound} we give the upper bound, which follows from an Assouad embedding-type argument in the case of slit carpets. Here, we follow the proof in \cite{heinonen2001lectures}. The more difficult lower bound requires a different approach based on estimating the growth of a scale-dependent distortion constant for an embedding of a finite approximation $\mathbb{M}_n$. Here we use methods and intuition coming from earlier works employing a similar strategy \cite{lang2001bilipschitz, burago1998separated, MR1616159, MR3816520}. 
The key idea in these arguments is usually some sort of dichotomy-result, which states that any would-be bi-Lipschitz embedding is forced to introduce extra stretching in two different ways. Then repeating this over finer and finer scales causes the distortion to increase as $n$ increases. In our case, we strengthen and adapt the techniques from \cite{seb2020regular} to give the required iteration. We have to skip a certain number of scales and for this reason we lose a logarithmic factor in the lower bound. It is an intriguing open problem, if this is actually necessary.

\subsection{$L^1$ embeddability of the slit carpet}
On the second part, we turn our attention to the $L^1$-embeddability of the slit carpet. Such embeddings have been extensively studied in particular in relation to the sparsest cut problem \cite{ebnilpotent, cheegerkleinerann,naoryoung, CKN, ALA, LN}. Curiously, these embedding results also have a close connection to the so-called Lipschitz dimension of Cheeger and Kleiner \cite{davidliplight, cheegerkleiner}, which remains a poorly studied notion of dimension on metric spaces. Using this notion of Lipschitz dimension we are able to prove our second main result:
\begin{theorem}\label{thm:embeddingL1}
	$ \mathbb{M}_{\infty} $ admits a $ (1,K) $-bi-Lipschitz embedding into $ L^1([0,1]) $.
\end{theorem}

This result follows from showing that the $\mathbb{M}_\infty$ has a Lipschitz dimension of 1, as defined by Cheeger and Kleiner \cite{cheegerkleiner}. Then, by using their general embedding result \cite[Theorem 1.1.]{cheegerkleiner} as a black-box tool, we find that $\mathbb{M}_\infty$ admits a bi-Lipschitz embedding into the space $L^1$. Contrast this to the discussion below, where we show that $\mathbb{M}_\infty$ does not embed to $\ell_1$. A similar phenomenon was known to occur for many inverse limits, see \cite{cheegerkleiner} for further discussion.

\subsection{Bi-Lipschitz embeddings of $\mathbb{M}^a$}
Next, in the section \ref{betaNumberChap}, we study the embeddability of dyadic slit carpets, denoted by $\mathbb{M}^a$, where the slits have smaller size: instead of the standard slit carpet with slit length of $2^{-n-1}$, the slits on level $n$ in $\mathbb{M}^a$ have length $2^{-n-1}a_n$, where $a_i = \frac{1}{n_i}$ for some $n_i \in \mathbb{N} \backslash \{0\}$. Our goal is to understand how the embeddability behaves when the slits become smaller and smaller, and as a consequence the dyadic slit carpet no longer has slits at every scale and location. In particular, blow-ups of such a space are locally Euclidean, and thus different methods are needed.

We obtain a necessary condition for the existence of a bi-Lipschitz embedding of the dyadic slit carpets $\mathbb{M}^a$, and this condition is determined by the sequence $a=(a_i)_{i\in \N}$:
\begin{theorem}\label{theorem2}
	If $\mathbb{M}^{a}$ admits a bi-Lipschitz embedding into Euclidean space, then $a\in \ell_{1 + \epsilon}$ for every $\epsilon>0$.
\end{theorem}

The main innovation in proving this result is the application of techniques from Jones' traveling salesman theorem and in particular the novel usage of $\beta$-numbers in proving non-embeddability results. $\beta$-numbers quantify how linearly a function behaves at most scales and locations, and we exploit the fact that a bi-Lipschitz embedding will have to behave linearly, at most scales and locations along rectifiable curves.

The proof begins by proving a dichotomy result for $L$-Lipschitz embeddings of slits. This dichotomy gives a lower bound for the $\beta$-numbers of those embeddings at certain scales, and combining this lower bound with Jones' traveling salesman theorem, we obtain Theorem \ref{theorem2}.

\subsection{Bi-Lipschitz embeddings of generalized slit carpets}
In the final part of the paper, we will study whether the locations of the slits affect the embeddability of the carpet. We will generalize the nonembeddability of the standard slit carpet by considering slit carpets where the slits can have any location and any (bounded) size. This result generalizes the qualitative main result in \cite{seb2020regular} in two ways: first, choosing the locations of the slits arbitrarily does not make the carpet embeddable, and secondly we show that the carpet cannot be embedded into any RNP space; see Definition \ref{def:rnp}.  
 We introduce the notion of generalized slit carpets $\mathbf{Z}$, where slits can be arbitrarily placed, but must occur at every scale and location and satisfy a mild $h$-condition limiting their positioning (see the section \ref{genSlitCarpetSect} for precise definition). 
\begin{theorem}\label{thm:RNPnonemb}
	A generalized slit carpet $\mathbf{Z}$ cannot be embedded into any Banach space with the RNP property, in particular, into any Euclidean space. 
\end{theorem}
For example, $\mathbb{M}_\infty$ is such a carpet, and $\ell_1$ is a Banach space with the RNP property. Thus, $\mathbb{M}_\infty$ does not bi-Lipschitz embed to $\ell_1$. In \cite{seb2020regular}, the authors showed that $\mathbb{M}_\infty$ does not bi-Lipschitz embed to any uniformly convex Banach space. The example of $\ell_1$ shows how RNP targets generalize this result. The carpets $\mathbb{M}^a$, however, are not included in Theorem \ref{thm:RNPnonemb}, since in them slits can decrease in size.

The key idea in the proof of this theorem is that for any Lipschitz map from $\mathbf{Z}$ to a RNP Banach space, the $y$-derivative of the map must be asymptotically constant. This uses Lebesgue differentiation and measurability. Thus, any nearby vertical curves are infinitesimally mapped to nearly parallel line segments. As a consequence, the midpoints of slits are mapped close together, which contradicts the Lipschitz bound on $f$.

\section{Notation and definitions} \label{nota}
\noindent  The distance on a metric space $X$ is denoted by $d_X(x,y)$ for $x,y \in X$. Where ambiguity does not arise, we will drop the subscript. Any path joining points $ x,y \in X$ is denoted by $ [x,y] $ (note that $[x,y]$ is also used to denote a closed interval in $\mathbb{R}$, but it will be clear from the context if it is an interval or a path) and its length is denoted by $ \ell([x,y]) $. A ball centered at $x$ with a radius of $r$ is denoted by $B(x,r):= \{y: d(x,y)<r\}$.

A subset $N \subset (X,d)$ is called $\epsilon$\textit{-separated}, if for every distinct $n,m\in N$ we have $d(n,m)\geq \epsilon$. A maximal $\epsilon$-separated set $N$ is called an $\epsilon$\textit{-net}. By maximality, every point of $X$ lies within a distance of $\epsilon$ from some element in $N$, i.e., the following holds: for every $x \in X$, there exists $y \in N$ such that $ d(x,y)<\epsilon$.

\begin{definition}\label{lips}
	We say that a mapping $ f:(X,d_X)\to(Y,d_Y) $ is \textit{$ (a,b) $-bi-Lipschitz mapping}, if
	\begin{align}\label{lipeq2}
		ad_X(x,y)\leq d_Y(f(x),f(y)) \leq b d_X(x,y),
	\end{align}
	holds for every $ x,y\in X $. If only the upper bound $d_Y(f(x),f(y)) \leq b d_X(x,y)$ holds, we say that $f$ is $b$-Lipschitz.

\end{definition}

For a given bi-Lipschitz embedding $ f:(X,d_X)\to(Y,d_Y) $, the infimum of all $ C $ for which there exists some scaling constant $r>0$ such that $f$ is $(r,Cr)$-bi-Lipschitz, is called the \textit{distortion of $ f $}, and denoted by $ \dist(f) $. The infimum distortion over all embeddings is called the \textit{optimal distortion} and denoted by $ c_{(X,Y)} $. We will mostly be interested in $c_{(X,\ell_2)}$, where $\ell_2$ is the usual infinite dimensional separable Hilbert space. More generally, $\ell_p=\{a=(a_i)_{i\in\N} : \sum_{i=1}^\infty  |a_i|^p<\infty\}$ are the standard sequence spaces for $p\in (1,\infty)$ and $L^1=\{f:[0,1]\to [-\infty,\infty] : \int |f|d\lambda < \infty\}$ is taken as the space of integrable functions on $[0,1]$. The space  $L^1$ is equipped with the usual norm and metric. Throughout, $\lambda$ will denote the Lebesgue measure on $\R$.

	A \textit{graph} $ G $ is a pair $G= (V,E) $ where $ V $ is the set of \textit{vertices} and $ E $ is a set of \textit{edges} connecting the vertices. 

\subsection{Construction of the standard slit carpets}\label{subsec:def_slit_carp}

First, let $\bM_0=M_0=[0,1]^2$. Then let $\cD_n$ be the collection of dyadic squares in $[0,1]^2$ at level $n\in \N$ given by 
\[
Q_{l,k}^n:=[l2^{-n},(l+1)2^{-n}]\times  [k2^{-n},(k+1)2^{-n}]
\]
for $l,k=0,\cdots, 2^{n}-1$. For each dyadic square $Q^n_{l,k}$ let $t^n_{l,k}:=((2l+1)2^{-n-1}, (4k+3)2^{-n-2})$ and $b^n_{l,k}=((2l+1)2^{-n-1}, (4k+1)2^{-n-2})$. We call the line segments $s_{l,k}^n:=[b^n_{l,k}, t^n_{l,k}]$ joining the two of these end points \emph{slits}. Using these, we define recursively sets $M_n$ as follows:
\[
M_n := M_{n-1} \setminus \bigcup_{l,k} s^n_{l,k},
\]
by cutting along the slits at level $n$. Then, we equip each open set $M_n$ with the path metric $d_n$, and complete the space to obtain $(\bM_n, d_{\mathbb{M}_n})$. The inclusion $p_{n,n-1}: M_{n}\to M_{n-1}$ induces a surjective $1-$Lipschitz map from the closures $\pi_{n,n-1}:\bM_n \to \bM_{n-1}$, and $\bM_\infty$ is the inverse limit of metric spaces
\[
\bM_0 \xleftarrow[\pi_{1,0}]{} \bM_1 \xleftarrow[\pi_{2,1}]{} \bM_2 \cdots \xleftarrow[\pi_{n,n-1}]{} \bM_n \xleftarrow[\pi_{n+1,n}]{} \bM_{n+1} \cdots \xleftarrow[]{}  \bM_\infty.
\]
The inverse limit $\bM_\infty$ consists of points represented by sequences $(x_k)_{k\in \N}, x_k\in \bM_k$ for which $\pi_{n,n-1}(x_n)=x_{n-1}$ for all $n\in \N$. The distance on the inverse limit is given by 
\[
d((x_k)_{k\in \N}, (y_k)_{k\in \N}) = \lim_{k\to \infty} d(x_k,y_k),
\]
where the limit exists since the projection maps $\pi_{n,n-1}$ are $1-$Lipschitz, and thus the sequence of distances is increasing.

The maps $p_{n,n-1}$ are injective, and thus the set $M_\infty = \bigcap_{n
	\in \N} M_n \subset [0,1]^2$ is injectively included in the inverse limit $\bM_\infty$. 
    Let $\pi_{n,0}:\mathbb{M}_n\to \mathbb{M}_0$ be the maps obtained as compositions $\pi_{n,0}=\pi_{1,0}\circ \cdots \circ \pi_{n,n-1}$ and as the inverse limit of such maps when $n=\infty$, i.e. $\pi_{\infty,0}((x_k)_{k\in \N})=x_0$.
If $p$ is a point in $\mathbb{M}_n$, then we denote the $x$- and $y$-coordinates of $\pi_{n,0}(p)$ by $p_x$ and $p_y$. Similarly, define $\pi_{n,m}$ for all $m<n$ with $n,m\in \N$ (or $n=\infty$).

In the next section, we will show how to bi-Lipschitz embed $\mathbb{M}_n$ into Euclidean space so that the distortion is controlled from above by $C\sqrt{n}$. For this purpose, we need the following lemma.
\begin{lemma}\label{projection}
	Let $i=1,...,n-1$. 	The projection maps $\pi_{n,i}:\mathbb{M}_n\to\mathbb{M}_i$ are $1$-Lipschitz and satisfy the quantitative lower bound
	\begin{equation}\label{eq:quant-lower}
		d_{\mathbb{M}_i}\big(\pi_{n,i}(x),\pi_{n,i}(y)\big)\ge \frac{d_{\mathbb{M}_n}(x,y)-2\cdot 2^{-i}}{2}
		\qquad\text{for all }x,y\in\mathbb{M}_n.
	\end{equation}
	In particular, whenever $d_{\mathbb{M}_n}(x,y)\ge 4\cdot 2^{-i}$ we have
	\[
	d_{\mathbb{M}_i}\big(\pi_{n,i}(x),\pi_{n,i}(y)\big)\ge 2^{-2}\,d_{\mathbb{M}_n}(x,y).
	\]
	
\end{lemma}
\begin{proof}
	
	Since $ \pi_{n,i}$ is a $1$-Lipschitz, we have
	$$d_{\mathbb{M}_i}(\pi_{n,i}(x),\pi_{n,i}(y)) \leq d_{\mathbb{M}_n}(x,y).$$ Thus we need to show the lower Lipschitz bound to prove the bi-Lipschitz property.
	
	To obtain a lower bound we note that given two points $x,y \in \mathbb{M}_n$, there always exists dyadic squares $Q_x$ and $Q_y$ of side length $2^{-i}$ such that $x \in Q_x$ and $y \in Q_y$. Next, we will construct a "taxicab-path" $[x,y]:=\gamma$ connecting the points $x$ and $y$ in $\mathbb{M}_n$. Note that since the slits are parallel to the $y$-axis, we can connect the points $x$ and $y$ to points $x'$ and $y'$ on the top side of $Q_x$ and $Q_y$ with straight line paths $[x,x']$ and $[y,y']$, respectively. Denote the shortest path connecting $\pi_{n,i}(x')$ and $\pi_{n,i}(y')$ by $\gamma'$. 
	Notice that for any square $Q_j$ the path $ \gamma'$ intersects, it connects two points on the circumference of $Q_j$. There exists a path traveling along the edges of $Q_j$ connecting these same two points that is at most twice the distance of the subsegment along $\gamma'$ connecting the same two points. Form $\gamma''$ in $\mathbb{M}_i$  by concatenating all of these detours along the grid set 
	$$\{ l2^{-i} \times [0,1]  \subset \mathbb{M}_n: l=0,\cdots, 2^{i}-1\} \cup \{ [0,1] \times  l2^{-i}  \subset \mathbb{M}_n: l=0,\cdots, 2^{i}-1\}.$$
    The path $\gamma''$ can be lifted isometrically to $\mathbb{M}_n$, since all the slits in $\mathbb{M}_k$ for $k=i+1,\dots, n+1$ occur within the dyadic squares of level $i$. This lifted path is at most twice as long as $\gamma'$ and thus $\ell(\gamma'')\leq 2 d_{\mathbb{M}_n}(x',y')$. We define $\gamma:= [x,x'] \cup \gamma'' \cup [y,y']$ and consider its projection onto $\mathbb{M}_{i}$. This yields
	\begin{align*}
		d_{\mathbb{M}_i}(\pi_{n,i}(x),\pi_{n,i}(y)) & \leq d_{\mathbb{M}_n}(x,y) \leq \ell(\gamma)\\
		& \leq  2 \cdot d_{\mathbb{M}_i}(\pi_{n,i}(x'),\pi_{n,i}(y'))+2\cdot2^{-i}
	\end{align*}

	Rearranging this estimate gives the \eqref{eq:quant-lower}. For the latter claim, note that if $d_{\mathbb{M}_n}(x,y)\geq 4\cdot 2^{-i}$, then  $d_{\mathbb{M}_n}(x,y)-2\cdot 2^{-i}\geq 2^{-1} d_{\mathbb{M}_n}(x,y)$, and the claim follows.
	
\end{proof}

\section{Upper bound for the Euclidean distortion in Theorem \ref{theorem1}}\label{sec:upperbound}

\noindent In this section, we will show how to construct a $(1,\sqrt{n}C)$-bi-Lipschitz embedding from $\mathbb{M}_n$ to a Euclidean space. We will need Theorem~3.2 from \cite{lang2001bilipschitz}, so we paraphrase it here without proof:
\begin{lemma}\label{gluing} Let $C\geq 1$, $n_A,n_B\in \N$. There exists a $D\geq 1$ s.t. the following holds. 
	Suppose $A,B \subset (X,d) $ and that there exist $(1,C)$-bi-Lipschitz embeddings $f_A : A \to \mathbb{R}^{n_A}$ and $f_B : B \to \mathbb{R}^{n_B}$. Then there exists a $(1,D)$-bi-Lipschitz embedding $f:A\cup B \to \mathbb{R}^{n_A + n_B +1}$.
\end{lemma}

The upper bound in Theorem \ref{theorem1} is given here.

\begin{theorem}\label{upper_bound}
	Suppose $ \mathbb{M}_n $ is the level $ n $ approximation of Merenkov's slit carpet. Then there exist constants $C,L\geq 1, L\in \N$, which are independent  of $ n\in \N$, and a $(1,C\sqrt{n})$-bi-Lipschitz mapping $ G: \mathbb{M}_n \to \mathbb{R}^L.$
\end{theorem}

\begin{proof} 
	
	Let $K\in \N$ be determined and let $L=(n-1)K$. We will construct a mapping $ G:\mathbb{M}_{n} \to \mathbb{R}^{L}$ in the form
	\begin{align*}
		G(x)=(g_{S_i^j} \circ \pi_{n,j})_{i=1,\dots, K, j=1, \dots n}.
	\end{align*}
	 The mappings $\pi_{n,i}$ in the above are the projections defined in Subsection ~\ref{subsec:def_slit_carp} and the following proof first focuses on the construction of the $c$-Lipschitz mappings $g_{S_k^i}:\mathbb{M}_i\to \R^{n_i}$. After this, it is immediate that the map $G$ is $\sqrt{Kn}c$-Lipschitz, and thus it suffices to bound $K$ and $c$, as well as to give a lower bound for $|G(x)-G(y)|$ for $x,y\in \mathbb{M}_n$.
	
	Fix $i\in \N$. First, take a $ \frac{1}{2^{i+2}} $-net $N_i \subset \mathbb{M}_i$ for every $ i=1,...,n $. Since $\mathbb{M}_i$ is doubling, $N_i$ can be partitioned into $K$ parts denoted by $ S_k^i $ ($k=1,...,K$) such that the following property holds: if $a,b \in S_k^i$ with $a\neq b$, then $d_{\mathbb{M}_i}(a,b)\geq \frac{1}{2^{i-7}}$. By the properties of a net, if $ x\in\mathbb{M}_i $, then there exists $ q \in S_k^i $ for some $ 1\leq k \leq K $ such that $ d_{\mathbb{M}_i}(x, q)<\frac{1}{2^{i+2}} $.
    This partitioning is obtained by a standard coloring argument; cf. \cite[p. 101]{heinonen2001lectures}. We recall the main steps here, and refer to there for more details. Form a graph $G_i$ whose vertices are $N_i$ and attach an edge between two vertices if their distance is at most $2^{7-i}$. This graph has degree at most a constant $D$ depending on the doubling constant. Then, employ a greedy coloring algorithm of assigning to each vertex, in order, a color different from all of its neighbors (whose colors were already assigned). Greedy coloring uses at most $K=D+1$ colors. This way no two vertices of the same color share an edge and for the colors $k=1,\dots, K$ collect to $S_k^i$ each vertex with color $k$.

	Notice that for any $x\in \mathbb{M}_i$, the set $B\left(x,\frac{1}{2^{i-5}}\right)$ intersects at most $M<\infty$ slits in $\mathbb{M}_i$ with $M\leq 2^{13}$. Indeed, this ball is contained in the union of at most $2^{12}$ dyadic squares of side length $2^{-i}$ and each such dyadic square can touch at most $2$ slits. Thus, we can divide $B\left(x,\frac{1}{2^{i-5}}\right) $ by vertical lines into $M+1$ pieces each of which is isometric to a rectangle in $\mathbb{R}^2$. Therefore, given any $ q_m \in N_i $, application of Lemma \ref{gluing} $M$ times shows that there exists a $ (1,c)  $-bi-Lipschitz embedding $ f_{q_m}:B\left(q_m,\frac{1}{2^{i-5}}\right)  \to \mathbb{R}^{L} $ such that $f_{q_{m}}(q_m)=0$ (we can translate the mapping if necessary). The same $c$ holds for every net $N_i$ and every $q_m \in N_i$, since $M$ holds uniformly for every $i$. 
	
	Then for each $k=1,...,K$ we define a map
	\begin{align*}
		g_{S_k^i}:\bigcup_{q_j \in S_k^i}B\left(q_j,\frac{1}{2^{i-5}}\right)\to \mathbb{R}^{L}
	\end{align*}
	such that $ g_{S_k^i}\big|_{B\left(q_j,\frac{1}{2^{i}}\right)}=f_{q_{j}} $ for every $ q_j \in S_k^i $. Notice that since $d_{\mathbb{M}_i}(q_j,q_l)\geq 2^{7-i}$ for every $j\neq l$  the different balls in this definition are disjoint and we have that the map $g_{S_k^i}$ is well defined. Since each $f_{q_{j}}$ is $(1,c)$-bi-Lipschitz,  $g_{S_k^i}$ is also $(1,c)$-bi-Lipschitz on $B\left(q_j,\frac{1}{2^{i}}\right)$ for every $q_j \in S_k^i$. 
	
	To see that $g_{S_k^i}$ is $2c$-Lipschitz in its whole domain, pick $x,y \in \mathbb{M}_i \cap \bigcup_{q_j \in S_k^i}B\left(q_j,\frac{1}{2^{i-5}}\right) $ such that there exists $ q_m\neq q_l $  such that $x \in B\left(q_m,\frac{1}{2^{i-5}}\right)$ and $y \in B\left(q_l,\frac{1}{2^{i-5}}\right)$. Since $d_{\mathbb{M}_i}(q_m,q_l)\geq \frac{1}{2^{i-7}}$ we get $ d_{\mathbb{M}_i}(y,x)\geq \frac{1}{2^{i-5}} $. Using this information and recalling that $f_{q_{m}}(q_m)=0$, the triangle inequality gives us
	\begin{align*}
		|g_{S_k^i}(x)-g_{S_k^i}(y)|&=|f_{q_{m}}(x)-f_{q_{l}}(y)|\leq|f_{q_{m}}(x)-f_{q_{m}}(q_{m})|+|f_{q_{l}}(q_{l})-f_{q_{l}}(y)|\\
		&\leq cd_{\mathbb{M}_i}(x,q_{m}) + c d_{\mathbb{M}_i}(y,q_{l}) \leq 2^{6-i}c\\
		&\leq 2cd_{\mathbb{M}_i}(x,y).
	\end{align*}
	Thus, each $g_{S_k^i}$ is $c$-Lipschitz on $ \mathbb{M}_i \cap \bigcup_{q_j \in S_k^i}B\left(q_j,\frac{1}{2^{i-5}}\right)$. Next, for each $ i=1,...,n $ and $ k=1,...,4^{i-3} $ we extend $ g_{S_k^i} $ using McShane's theorem \cite{MR1562984} to get $c_1$-Lipschitz maps defined on the whole space $ \mathbb{M}_i $. We keep denoting these extensions with the same symbol $ g_{S_k^i} $.  

    With the mappings constructed, we proceed to show a lower bound for $|G(x)-G(y)|$. Now pick $x,y \in \mathbb{M}_n$. Let $z \in \mathbb{M}_n$ be a point which satisfies $\frac{d_{\mathbb{M}_n}(x,y)}{2}=d_{\mathbb{M}_n}(z,y)=d_{\mathbb{M}_n}(z,x)$ and let $ l=2,\dots$ be such that $ 2^{2-l}\leq d_{\mathbb{M}_n}(x,y)\leq2^{3-l} $. 
	
	Suppose first that $l<n$. The condition $2^{2-l}\leq d_{\mathbb{M}_n}(x,y)$ and Lemma ~\ref{projection} gives
	
	\[ 2^{-2}\,d_{\mathbb{M}_n}(x,y)  \leq d_{\mathbb{M}_l}\big(\pi_{n,l}(x),\pi_{n,l}(y)\big)\leq d_{\mathbb{M}_n}(x,y)\leq 2^{3-l}. \]
	Since $N_l$ is a $2^{-i-2}$-net, there exists $q \in S_k^l$ such that $\pi_{n,l}(z) \in B\left(q,\frac{1}{2^{l+2}} \right)$, for some $k$, and consequently $\pi_{n,l}(x),\pi_{n,l}(y) \in B\left(q,\frac{1}{2^{l-5}}\right) $. Since $g_{S_k^l}$ is $(1,c)$-bi-Lipschitz on $B\left(q,\frac{1}{2^{l-5}}\right) \subset \mathbb{M}_l$, we have the lower bound
	\[ 
	2^{-2}\,d_{\mathbb{M}_n}(x,y)\leq |g_{S_k^l}(\pi_{n,l}(x)) - g_{S_k^l}(\pi_{n,l}(y))|\leq |G(x)-G(y)|.
	\]

	On the other hand, if $l \geq n$, then by similar argumentation as above, there exists $q \in S_k^n$, for some $k$, such that $z \in B\left(q,\frac{1}{2^{n+2}}\right)$ and $x,y \in B\left(q,\frac{1}{2^{5-n}}\right)$. Recall that $g_{S_k^n}$ is $(1,c)$-bi-Lipschitz on $B\left(q,\frac{1}{2^{5-n}}\right)$, and thus
    \[ 
	\,d_{\mathbb{M}_n}(x,y)\leq |g_{S_k^l}(\pi_{n,l}(x)) - g_{S_k^l}(\pi_{n,l}(y))|\leq |G(x)-G(y)|.
	\]
    This completes the proof of the lower bound and shows that $G$ is $(2^{-2},\sqrt{Kn}c)$-bi-Lipschitz, which yields the desired claim by scaling the map by $4$.
	
\end{proof}

	Now we have shown that $ c_{(\mathbb{M}_n,\ell_2)} \leq C\sqrt{n}$, for some uniformly holding fixed $ C $.  Notice that $c_{(\mathbb{M}_n,\ell_2)}$ must be increasing in $n$, and in the next section we will quantify this increasing behavior from below.
	
	\begin{lemma}\label{incdis}
		The optimal bi-Lipschitz distortion of $ \mathbb{M}_n $ is increasing in $ n $.
	\end{lemma}
	\begin{proof}

		Let $f$ be an arbitrary bi-Lipschitz Euclidean embedding of $\mathbb{M}_{n+1}$ and set $A:=\left[ \frac{1}{2},1\right] \times\left[ \frac{1}{2},1\right] \cap \mathbb{M}_{n+1} $. We have $ \dist(f|_A)\leq\dist(f) $, where $f|_A$ is the restriction of $f$ to $A$. But $A$ is just a scaled version of $\mathbb{M}_n$ and distortion is scale invariant. Thus, $c_{(\mathbb{M}_n,\ell_2)}\leq c_{(\mathbb{M}_n,\ell_2)}$ since distortion decreases upon restriction.
		
	\end{proof}

	\section{Lower bound for the Euclidean distortion in Theorem \ref{theorem1}}
	\noindent The Laakso graph shares several structural properties with the standard slit carpet; see Figure~\ref{fig:laakso} below. There are several known proofs of lower bounds in the Laakso graph setting; see, for instance, Lang and Plaut \cite[Theorem 2.3]{lang2001bilipschitz}, where a lower bound of $\sqrt{n}$ is obtained. Our approach follows the same general strategy. The Laakso construction is simpler, since the diamonds are nested across scales. In contrast, in the standard slit carpet the corresponding slits are not nested but lie adjacent to one another, and this lack of nesting constitutes the main technical difficulty.
	
	\begin{figure}[h]
		\centering
	\def\svgwidth{0.5\linewidth}
	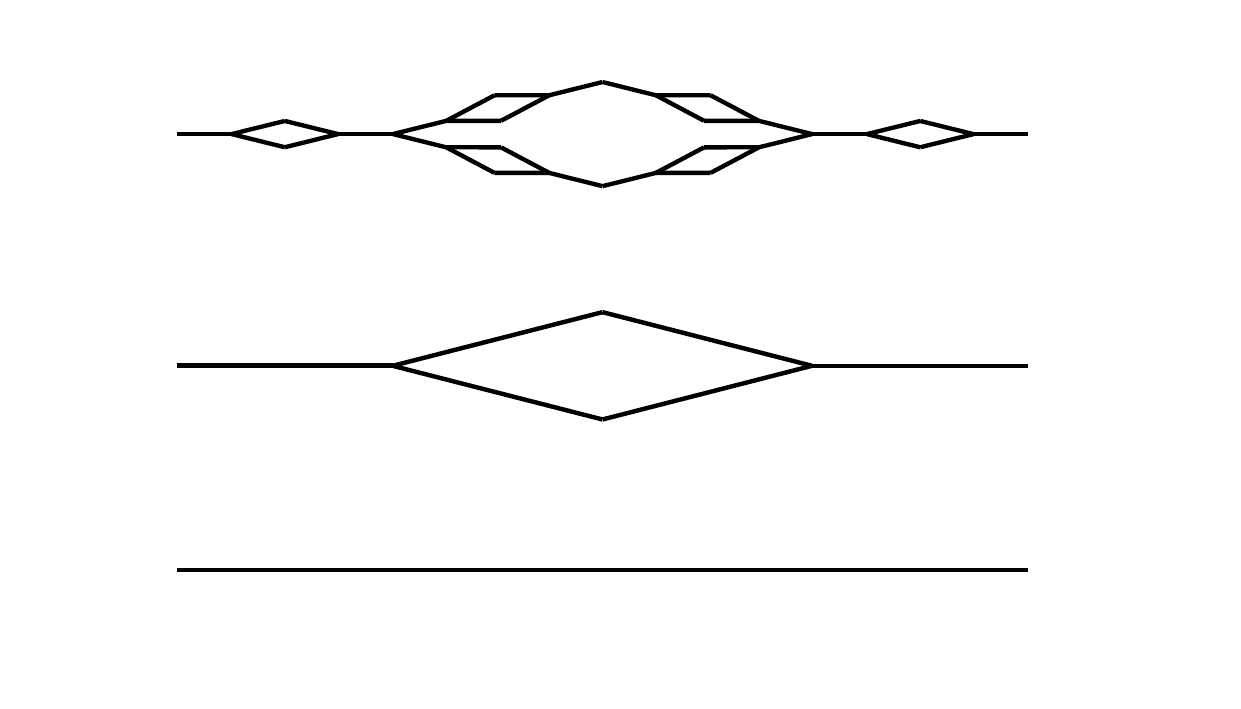

		\caption{First three iterations of the Laakso graph.}
		\label{fig:laakso}
	\end{figure}

	To address this issue, we introduce an argument that forces the analysis to skip multiple scales. For $\mathbb{M}_n$, this requires skipping on the order of $\log n$ scales. While this step is essential for the argument to go through, it necessarily weakens the resulting estimate. As a consequence, we are unable to recover the $\sqrt{n}$ lower bound known for Laakso graphs and instead obtain a weaker lower bound for the optimal distortion $c(\mathbb{M}_n,\ell_2)$.
	
	At present, we do not know whether this bound is sharp. It is plausible that it can be improved, possibly by iterated logarithmic factors. The scale--skipping argument may be viewed as a quantitative refinement of the qualitative techniques developed in \cite{seb2020regular}.
	
	\begin{theorem}\label{lowerbound}
		For any bi-Lipschitz embedding $f: \mathbb{M}_n \to \mathbb{R}^K$, we have $ \dist(f) \geq C \sqrt{\frac{n}{\log(n)}}$, where $C$ is a constant independent of $n$.
	\end{theorem}
	\begin{proof}
	Let $f:\mathbb{M}_n\to \R^K$ be $(1,D)$-bi-Lipschitz, and for $k \geq 0 $ define
	\[
	L_{k,n} := \sup_{y,y' \in \mathcal{V}_{k,n}} \frac{|f(y)-f(y')|}{d_{\mathbb{M}_n}(y,y')},
	\]
	where 
	\begin{align*}
		\mathcal{V}_{k,n}
		:= \Big\{ (y,y') \in \mathbb{M}_n \times \mathbb{M}_n :\ 
		&\exists\,  m \in \{0,\dots,2^{k}-1\},\ r \in \{0,\dots,2^{k}-1\}\ \text{such that} \\[4pt]
		&\pi_{n,0}(y) = \Big( \tfrac{2m+1}{2^{k+1}},\, r2^{-k} \Big),\quad 
		\pi_{n,0}(y') = \Big( \tfrac{2m+1}{2^{k+1}},\, (r+1)2^{-k} \Big)
		\Big\}.
	\end{align*}
	This definition means that the set $\mathcal{V}_{k,n}$ contains every pair of middle points of the top and bottom of level $k$ dyadic squares in $[0,1]^2$. 
	Define
	\[
	L_x := \sup_{x,x'\in \mathcal{H}_n} \frac{|f(x) -f(x')|}{d_{\mathbb{M}_n}(x,x')},
	\]
	where $\mathcal{H}_n$ is the collection of $(x,x') \in \mathbb{M}_n \times \mathbb{M}_n$ s.t. there exists a $\Delta\in \R$ with $\pi_{n,0}(x') = \pi_{n,0}(x) + \Delta e_1$ and the segment $[\pi_{n,0}(x),\pi_{n,0}(x')]$ does not intersect the interior of a slit.
    Here $e_1=(1,0)$ is the standard unit basis vector. See Figure ~\ref{fig:lipdis} for an illustration of these points.
	
	\begin{figure}[h]
		\centering
	\def\svgwidth{0.5\linewidth}
	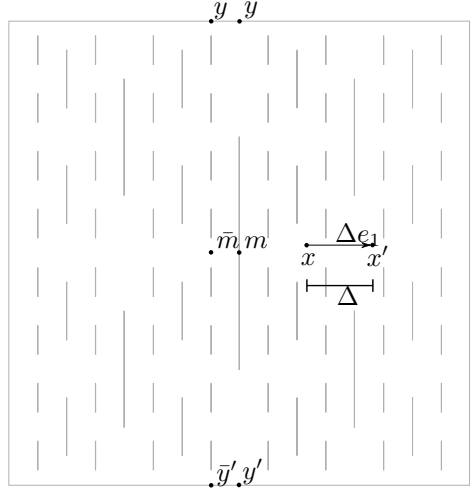

		\caption{Illustration of $L_x$ and $L_{k,n}$.}
		\label{fig:lipdis}
	\end{figure}

	We will prove a lower bound of the form $C\sqrt{n/\log(n)}$ for $L_{l,n}$ or $L_X$ for some $l=1,\dots, n$. This gives a lower bound for $D$ and since $f$ is arbitrary, for the distortion of any bi-Lipschitz embedding. If $L_X \geq \sqrt{n}/16$ or $L_{l,n} \geq \sqrt{n}/16$ for some $l=1,\dots, n$, we are done. Thus, we may assume $L_X,L_{l,n} < \sqrt{n}$ for all $l=1,\dots, n$. Let $j=\lceil \log_2(n) \rceil$ and $k= \Bigl\lfloor\dfrac{n}{j}\Bigr\rfloor-1$. In the following, we restrict without loss of generality to $n\geq 8$ (in order to ensure $k\geq 2$). Let $l=1,\dots, k$. We will develop an estimate for $L_{(l+1)j}$ in terms of $L_{lj}$ by employing diamonds formed by slits, somewhat similar to the proof in \cite{lang2001bilipschitz}. 
	 
    	Choose $y,y' \in \mathcal{V}_{lj,n}$ so that $|f(y)- f(y')| = L_{lj,n}d_{\mathbb{M}_n}(y,y') $, and let $m$ and $m'$ be the two copies of the midpoint of the slit between $y$ and $y'$. Let $m$ be the midpoint on the left hand side of the slit. Note that $ d_{\mathbb{M}_n}(m,m')=\frac{1}{2}d_{\mathbb{M}_n}(y,y') $. We first do some preliminary estimates. 
	
	   By applying dihedral symmetries, we may assume $$|f(y)-f(m)|=\max\{|f(y)-f(m)|,|f(m)-f(y')|,|f(y')-f(m')|,|f(m')-f(y)|\}. $$ Then by the parallelogram law 
	\begin{align*}
		L_{lj,n}^2d_{\mathbb{M}_n}(y,m)^2 + \frac{1}{4}d_{\mathbb{M}_n}(y,m)^2&=L_{lj,n}^2\frac{1}{4}d_{\mathbb{M}_n}(y,y')^2 + \frac{1}{16}d_{\mathbb{M}_n}(y,y')^2\\
		&=\frac{1}{4} \left( L_{lj,n}^2d_{\mathbb{M}_n}(y,y')^2 +d_{\mathbb{M}_n}(m,m')^2 \right)\\
		& \leq \frac{1}{4} \left( |f(y)-f(y')|^2 + |f(m)-f(m')|^2 \right)\\
		& \leq \frac{1}{4} \left( |f(y)-f(m)|^2 + |f(m)-f(y')|^2 + |f(y')-f(m')|^2 + |f(m')-f(y)|^2 \right)\\
		& \leq |f(y)-f(m)|^2.
	\end{align*}
    This shows that the pair $(m,y)$ is stretched more than $(y,y')$. We will next shift this pair of points a bit to the left.	Consider $\bar{y}, \bar{m}$ obtained by shifting $y,m$ to the left by $2^{-(l+1)j-1}$; see Figure ~\ref{fig:lipdis}.

	With this choice of points the triangle inequality gives us

	\begin{align*}
		|f(\bar{y})- f(\bar{m})| &\geq |f(y)-f(m)| - L_X (d_{\mathbb{M}_n}(y, \bar{y}) + d_{\mathbb{M}_n}(m, \bar{m})) \\
		&\geq \sqrt{L_{lj,n}^2 + 1/4} d_{\mathbb{M}_n}(y,m) -L_X2^{2-j} d_{\mathbb{M}_n}(y,m) \\
		&\geq (\sqrt{L_{lj,n}^2 + 1/4} - L_X 2^{2-j})d_{\mathbb{M}_n}(\bar{y},\bar{m}).
	\end{align*}

	Next split the vertical segment between $[\bar{y}, \bar{m}]$ into $2^{j-1}$ sub-segments $\{[y_i,y_{i+1}]\}_{i=1}^{2^{j-1}}$, where each pair is contained in $\mathcal{V}_{(l+1)j}$. Since these partition the segment, a simple triangle inequality argument gives that there is at least one index $i$ s.t. 
    \begin{align*}
		|f(x_{i+1})- f(x_i)| &\geq (\sqrt{L_{lj,n}^2 + 1/4} - L_X 2^{2-j})d_{\mathbb{M}_n}(x_{i+1},x_i).
	\end{align*}
    We also have $|f(x_{i+1})- f(x_i)| \leq L_{(l+1)j-1,n}d_{\mathbb{M}_n}(x_{i+1},x_i) $ by the definition of $ L_{(l+1)j,n}$. Squaring these, we get 

	\begin{align*}
		L_{(l+1)j,n}^2 d_{\mathbb{M}_n}(x_{i+1},x_i)^2 \geq |f(x_{i+1})- f(x_i)| &\geq  (\sqrt{L_{lj,n}^2 + 1/4} - L_X 2^{2-j})^2d_{\mathbb{M}_n}(x_{i+1},x_i)^2
	\end{align*}
    Thus,
    \[
    L_{(l+1)j,n}^2 \geq L_{lj,n}^2 + 1/4 -2L_X 2^{2-j}\sqrt{L_{lj,n}^2 + 1/4}.
    \]
    Next, $L_X\leq \sqrt{n}/16$, $2^{-j}\leq  n^{-1}$ and $\sqrt{L_{lj,n}^2 + 1/4}\leq 2L_{lj,n}\leq \sqrt{n}/16$. Thus
	\[
    L_{(l+1)j,n}^2 \geq L_{lj,n}^2 + 1/8.
    \]
	Iterating the above inequality $k$ times we obtain,
	\[
		\frac{k}{8} \leq L_{kj,n}^2\leq D^2
	\]

	and thus $ D\geq \frac{1}{4}\sqrt{k}\geq\frac{1}{8}\sqrt{\frac{n}{j}}\geq\frac{1}{8}\sqrt{\frac{n}{\log_2(n)}}$, which yields the claim.
\end{proof}

Finally, we observe that Theorem \ref{theorem1} now immediately follows from Theorems \ref{upper_bound} and \ref{lowerbound}.

	\section{Embedding the carpet into $L^1$}
	
	\noindent In the previous sections we studied embeddings of $\mathbb{M}_\infty$ into Hilbert space (in particular into $\ell_2$). In this fifth section we prove that the standard slit carpet $\mathbb{M}_\infty$ admits an embedding into $L^1$. Our strategy is to construct a \emph{Lipschitz--light} map from $\mathbb{M}_\infty$ to $\mathbb{R}$ and then invoke the main Cheeger--Kleiner implication which yields an embedding to $L^1$.
	
	\subsection*{Definition of Lipschitz light}
	 Lipschitz light mappings were originally defined in Cheeger--Kleiner \cite{cheegerkleiner}. We adopt the definition of a Lipschitz light map from David \cite[Definition 1.2]{davidliplight}. A discrete path $(p_0,\dots,p_n)$ in a metric space $(X,d)$ is a $\delta$-path if
	\[
	d(p_i,p_{i+1})\le \delta\qquad(0\le i\le n-1).
	\]
	A subset $A\subset X$ is $\delta$-connected if for every $a,b\in A$ there exists a $\delta$-path $(p_0,\dots,p_n)$ with $p_0=a$, $p_n=b$ and $p_k\in A$ for all $k$. A $\delta$-component of $U\subset X$ is a maximal $\delta$-connected subset of $U$.
	
	\begin{definition}\label{def:liplight}
		Let $f:X\to Y$ be a map between metric spaces. We say $f$ is \emph{$C$-Lipschitz--light} if there exists $C\ge1$ such that
		\begin{enumerate}
			\item $f$ is $C$-Lipschitz, and
			\item for every $r>0$ and every $W\subset Y$ with $\diam(W)\le r$, each $r$-component of $f^{-1}(W)$ has diameter $<Cr$.
		\end{enumerate}
	\end{definition}

	\subsection*{Cheeger--Kleiner implication}
	Cheeger and Kleiner showed a key result, which we will apply as a black-box tool:
	
	\begin{theorem}[Cheeger--Kleiner, Theorem 1.2 in \cite{cheegerkleiner}]  \label{thm:ck-main}
		If a metric space $X$ admits a Lipschitz--light map into $\mathbb{R}$, then $X$ admits a bi-Lipschitz embedding into some $L^1$-space.
	\end{theorem}
	Thus, it suffices to construct a Lipschitz--light map $f:\mathbb{M}_\infty\to\mathbb{R}$.
	
	We will use the coordinate projections $\pi_{\infty,k}:\mathbb{M}_\infty\to\mathbb{M}_k$ given by $\pi_{\infty,k}((x_i)_{i\ge0})=x_k$. Abbreviate these maps as $\pi_k,\pi_0$. Each $\pi_k$ is $1$-Lipschitz; in particular we write $\pi:=\pi_0:\mathbb{M}_\infty\to\mathbb{M}_0\cong Q_0$. 
    
    A map $f:X\to Y$  is called ($C$-)regular, for $C\in \N$, if it is $C$-Lipschitz and for every $y\in Y$ and $r>0$ there exists $x_1,\dots, x_C \in X$ s.t. $f^{-1}(B(x,r))\subset \bigcup_{j=1}^C B(x_i, r)$.   The regularity of the map $\pi$ is established in \cite[Lemma 2.3]{merenkov}.
	
	\begin{lemma}[Merenkov]\label{lem:pi-regular}
		The projection $\pi=\pi_0:\mathbb{M}_\infty\to\mathbb{M}_0$ is regular.
	\end{lemma}
 The $x$- and $y$-components of $\pi$ are not Lipschitz--light because some of their fibres contain nontrivial curves that produce large $r$-components. The key geometric idea is to \emph{tilt} the projection by a small slope: compose $\pi$ with a linear functional of small slope.
 
	\subsection*{Tilted projection and the main theorem}

	Fix $\alpha\in(0,1/4)$ and define
	\[
	\Lambda_\alpha(x,y)=\alpha x + y\qquad(\Lambda_\alpha:\mathbb{R}^2\to\mathbb{R}).
	\]
	Set $f=\Lambda_\alpha\circ\pi:\mathbb{M}_\infty\to\mathbb{R}$. The main technical theorem of this section is:
	
	\begin{theorem}\label{thm:liplight}
		There exists $C\ge1$ (depending only on constants from the construction and the regularity constant of $\pi$) such that for every $\alpha\in(0,1/4)$ the map $f=\Lambda_\alpha\circ\pi$ is $C$-Lipschitz--light.
	\end{theorem}

	Combining Theorem \ref{thm:liplight} with Theorem \ref{thm:ck-main} yields  immediately the proof of Theorem \ref{thm:embeddingL1}. To prove Theorem \ref{thm:liplight} we analyze the geometry of the tilted preimage
\(\widetilde W=\Lambda_\alpha^{-1}(W)\cap[0,1]^2\) of a small interval \(W\subset\mathbb R\).
Two elementary but crucial ingredients are needed for this analysis. The first is a
purely geometric fact which shows that a vertical cut through a narrow parallelogram
forces any continuous path joining the two sides to have a definite minimal length;
the second is a short combinatorial lemma that controls gaps in certain modular
sequences and hence limits how many slits can be missed by $\widetilde W$. Together, these lemmas allow us to cover \(\widetilde W\) by finitely many
parallelograms whose lifts under the regular projection \(\pi\) are well separated
(at scale \(r\)), and then to bound the diameter of each \(r\)-component of
\(f^{-1}(W)\) by a constant multiple of \(r\). We state the two lemmas used in the
argument below.

\begin{lemma}\label{lem:parallelogram}
	Let $\alpha,r\in\mathbb{R}$ and let $x_0\in\mathbb{R}$ be fixed.  Let
	\[
	R=\{(x,y): x\in[c,d],\; x_0-\alpha x \le y \le x_0-\alpha x + r\}
	\]
	be a parallelogram bounded by the lines $y=x_0-\alpha x$ and $y=x_0-\alpha x + r$ and the vertical lines $x=c$ and $x=d$.  Let $w\in(c,d)$ and let $\delta_+,\delta_->0$.  Consider a vertical segment
	\[
	l_w=\{w\}\times [\,x_0-\alpha w-\delta_-,\; x_0-\alpha w + r+\delta_+\,],
	\]
	so that $l_w$ meets $R$ and splits $R$ into the left and right parts
	\[
	R^L=\{(x,y)\in R:x<w\},\qquad R^R=\{(x,y)\in R:x>w\}.
	\]
	Then in the plane $\mathbb{R}^2\setminus l_w$ any continuous path joining a point of $R^L$ to a point of $R^R$ has length at least $\min(\delta_-,\delta_+)$.
\end{lemma}

\begin{proof}
	Let $z^L \in R^L$ and $z^R \in R^R$ and let
	\[
	\gamma : [0,1] \to \mathbb{R}^2 \setminus l_w
	\]
	be any continuous path with $\gamma(0)=z^L$, $\gamma(1)=z^R$. The vertical segment
	\[
	l_w = \{w\} \times [a,b]
	\]
	with
	\[
	a := x_0 - \alpha w - \delta_-, \qquad
	b := x_0 - \alpha w + r + \delta_+
	\]
	is removed from the plane. Since $\gamma$ connects a point with $x<w$ to a point
	with $x>w$, by continuity there exists $t \in (0,1)$ with
	\[
	\gamma(t) = (w,y_w).
	\]
	Because $\gamma(t) \notin l_w$, we must have either
	\[
	y_w \ge b = x_0 - \alpha w + r + \delta_+
	\quad\text{or}\quad
	y_w \le a = x_0 - \alpha w - \delta_-.
	\]
	
	Assume the first case $y_w \ge x_0 - \alpha w + r + \delta_+$ (the other case is
	analogous with $\delta_-$). Write
	\[
	z^R = (x_R,y_R)
	\]
	with $x_R > w$. 
	If $\alpha \ge 0$ then the map $x \mapsto x_0 - \alpha x$ is decreasing, hence
	\[
	y_R \le x_0 - \alpha x_R + r \le x_0 - \alpha w + r.
	\]
	
	Define the orthogonal projection
	\[
	p = (w,y_R).
	\]
	Then the triangle with vertices $z^R$, $p$, and $(w,y_w)$ is right-angled at $p$, so
	\[
	d\bigl(z^R,(w,y_w)\bigr)
	\ge d\bigl(p,(w,y_w)\bigr)
	= |y_w - y_R|
	\ge y_w - (x_0 - \alpha w + r)
	\ge \delta_+.
	\]
	
	Thus the portion of $\gamma$ between $z^R$ and $(w,y_w)$ has length at least
	$\delta_+$, and consequently the total length of $\gamma$ is at least $\delta_+$.
	The case $\alpha<0$ is handled the same way (the monotonicity inequalities reverse,
	so one uses $z^L$ instead of $z^R$). The argument for $y_w \le a$ is identical with
	$\delta_-$ in place of $\delta_+$. Therefore every path in
	$\mathbb{R}^2 \setminus l_w$ from $R^L$ to $R^R$ has length at least
	$\min(\delta_-,\delta_+)$.
\end{proof}

In fact, this result is sharp in the following sense: When $\alpha\to\infty$, $d(z^R,(w,x_0-\alpha x + r))\to 0$ and $d(z^L,(w,x_0-\alpha x + r))\to 0$, then $d(z^R,z^L)\to\delta_{+}$.

Next, we give our second tool, a simple combinatorial lemma, which is needed to control the gaps in intersections with certain intervals in arithmetic sequences. For $a>0$ we denote by $(x \ {\rm mod}\  a) \in [0,a)$ the remainder under division by $a$, for $x>0$. 

\begin{lemma}\label{lem:combinatorics}
	Let $x_0,a>b>0$,  Let $I\subset [0,a)$ be a nonempty interval and $b<|I|$. Let $x_n := (x_0 + n b) \ {\rm mod} \  a$ and define $B_I:=\{n\in \mathbb{N}: (x_n {\rm \ mod\  a}) \in I  \}$.  so that for every $n,m\in \mathbb{N}$ with $[n,m]\cap B_I=\emptyset$, we have $|n-m|\leq \lceil a/b\rceil + 1$.
\end{lemma}

\begin{proof}
	Let $C=\lceil a/b\rceil + 1$, and let $I=[e,f]$ (so $0\le e<f<a$). It suffices to show that for every $n \in \mathbb{N}$, at least one of the elements in $\{ n,n+1,\dots, n+(C-1)\}$ lies in $B_I$. Suppose this were not the case. Then
	\[
	x_n,x_{n+1},\dots, x_{n+(C-1)}\not\in I.
	\]
	Let $\tilde{I}=a\mathbb{Z} + I$, and let $y_n = x_0 + n b$. We have that $y_{n+k} \not\in \tilde{I}$ for $k=0, \dots, C-1$. Let $z\in \mathbb{Z}$ be such that $y_n \in [za,(z+1)a)$. Since $y_n\not\in \tilde{I}$, we have $y_n \in [za,za+e) \cup (za+f,(z+1)a)$.
	
	If $y_n\in [za,za + e)$, then $za+e-y_n \le za+e-za \le e \le a$, and since $b<f-e$ (because $b<|I|=f-e$), for
	\[
	T=\left\lceil \frac{za + e-y_n}{b}\right\rceil
	\]
	we have
	\[
	T \le \left\lceil \frac{a}{b}\right\rceil = C-1,
	\]
	i.e. $(T-1)b<za + e-y_n \leq Tb $ which with $b<|I|$ yields $y_{n+T} \in [za +e, za + f] \subset \tilde{I}$, contradicting the assumption. Thus $y_n\in [za,za + e)$ is not possible.

	On the other hand, if $y_n \in (za+f,(z+1)a)$, then the distance from $y_n$ to the next copy of $za+e$ (i.e. $(z+1)a+e$) is at most $a$:
	\[
	(z+1)a + e - y_n \le a + e - f \le a,
	\]
	and the same argument as above with
	\[
	T=\left\lceil \frac{(z+1)a+e-y_n}{b}\right\rceil \le \left\lceil\frac{a}{b}\right\rceil = C-1
	\]
	yields $y_{n+T}\in [(z+1)a+e,(z+1)a+f] \subset \tilde{I}$, again a contradiction.
	
	This proves that every block of $C$ consecutive indices meets $B_I$, which is equivalent to the stated bound on gaps.
\end{proof} 

Next we will prove the main theorem of this section. Figure~\ref{fig:lipschitz_light} below illustrates the main idea of the proof in a simplified way for selected parameters.
\begin{figure}[h]
	\centering
	\includegraphics[width=0.8\textwidth]{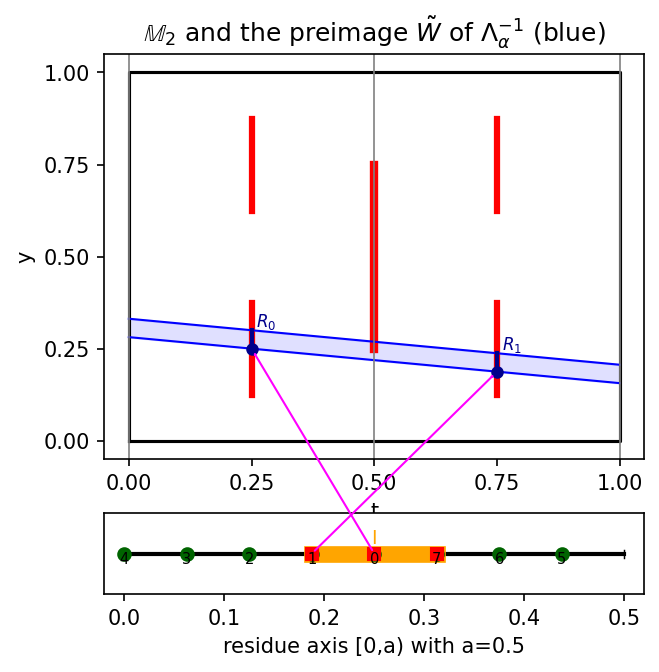}
	\caption{Schematic relation between the residue sequence and the tilted strip.  The \textbf{upper panel} shows the unit square with vertical slits (red) and the tilted strip
		\(\tilde W=\Lambda_\alpha^{-1}([x_0,x_0+r])\cap[0,1]^2\) (blue band) bounded by the lines
		\(y=x_0-\alpha t\) and \(y=x_0-\alpha t + r\). Dark-blue dots mark sampled midpoints
		\(t=2^{-k-1}+n2^{-k}\) and the short dark-blue vertical segments indicate the vertical pieces
		of \(\tilde W\) used in the parallelogram covering. The \textbf{lower panel} displays the residue
		axis \([0,a)\) with \(a=2^{-k}\): green dots are the residues \(x_n=(x_0' + n b)\bmod a\),
		the orange bar is the interval \(I=[3\cdot2^{-k-3},5\cdot2^{-k-3}]\), and red squares highlight
		residues lying in \(I\). Magenta arrows are schematic and illustrate how residues \(x_n\in I\)
		correspond to vertical pieces at the associated \(t\)-midpoints. Parameters used for the sketch:
		\(\alpha=1/8,\; r=0.05,\; k=1\). }
	\label{fig:lipschitz_light}
	\vspace{4pt}
\end{figure}

\begin{proof}[Proof of Theorem \ref{thm:liplight}]
	As a composition of a linear map with the $1$-Lipschitz $\pi_0$, the map $\Lambda_\alpha$ is Lipschitz. Thus, it remains to prove the second condition in Definition \ref{def:liplight}. Let $r>0$ and $W\subset \mathbb{R}$ be a set with $\diam(W)\leq r$. Without loss of generality take $W=[x_0,x_0+r] \subset [0,2]$.  Let $\widetilde{W} = \Lambda_\alpha^{-1}([x_0,x_0+r])\cap [0,1]^2$.
	
	First, suppose $r\geq 2^{-3}$. Each component of $f^{-1}(W)$ has diameter at most $\diam(\mathbb{M}_\infty)\leq 3\leq 24 r$. Thus, for such $r$, the $C$-Lipschitz light condition holds for any $C\geq 24$.
	
	Next, consider $r<2^{-3}$, and let $k\geq 0$ be a natural number so that $2^{-k-4}\leq r<2^{-k-3}$. Let
	\[
	T = \{t : x_0 -\alpha t \in [0,1]\}.
	\]
	Also, let $N = \{n\in\mathbb{Z} : [n2^{-k}, (n+1)2^{-k}]\cap T \neq \emptyset\}$ and hence $T \subset \bigcup_{n \in N} [n2^{-k}, (n+1)2^{-k}]$. We have
	\[
	\widetilde{W}\subset\bigcup_{t\in T} \{t\}\times  [x_0-\alpha t, x_0-\alpha t + r].
	\]
	
	Put $a=2^{-k}$, $b=\alpha a$, and let
	\[
	I = \big[3\cdot 2^{-k-3},\,5\cdot 2^{-k-3}\big],
	\]
	and set $x_0' = x_0 - \alpha 2^{-k-1}$. As in Lemma \ref{lem:combinatorics} consider the set
	\[
	B_I = \{n \in \mathbb{Z} : x_0' + nb \in I \ {\rm mod}\ a\}.
	\]
	Then, by applying Lemma \ref{lem:combinatorics} and noting that $b<|I|=2^{-k-2}$ (which follows from the choice of $a$ and the assumption $\alpha<1/4$), we get that there exists a constant $C_1\in \mathbb{N}$ (independent of $k$) so that whenever $n,m\in B_I$ with $(n,m)\cap B_I =\emptyset$, then $|n-m|\leq C_1$.
	
	Let $\tilde{B}=B_I\cap N$, and enumerate the elements of this set in increasing order: $\widetilde{B}=\{b_1,\dots, b_s\}$. Since
	\[
	T\subset \bigcup_{n\in N} [n2^{-k}, (n+1)2^{-k}],
	\]
	we obtain the covering
	\[
	T \subset [\min(N) 2^{-k}, 2^{-k-1}+b_1 2^{-k}]\cup [2^{-k-1} + b_s 2^{-k}, (\max(N)+1)2^{-k}] \cup \bigcup_{i=1}^{s-1} [2^{-k-1} + b_i 2^{-k},2^{-k-1} + b_{i+1} 2^{-k}].
	\]
	
	Denote the intervals in the above union by $I_0,\dots, I_{s}$. By the choice of $C_1$, we have $b_1-\min(N)\leq C_1$, $\max(N)-b_s\leq C_1$ and $b_{i+1}-b_i\leq C_1$. Thus each of the intervals in the previous union has length at most $(C_1+1)2^{-k}\leq 16(C_1+1) r$ (since $2^{-k}\le 16 r$).
	
	Now, from the definition of $\Lambda_\alpha$ and the description of $\widetilde W$ we get
	\[
	\widetilde{W} \subset \bigcup_{i=0}^s \pi^{-1}\left(\bigcup_{t\in I_i} \{t\}\times  [x_0-\alpha t, x_0-\alpha t + r]\right).
	\]
	
	For each $i$ denote
	\[
	R_i=\bigcup_{t\in I_i} \{t\}\times  [x_0-\alpha t, x_0-\alpha t + r].
	\]
	It is easy to see that each $R_i$ is a parallelogram with diameter at most $17(C_1+1)r$. Further, the vertical sides of the parallelograms $R_i$, for $i=1,\dots, s-1$, are of the form
	\[
	\{2^{-k-1} + b_j 2^{-k}\} \times [x_0-\alpha (2^{-k-1} + b_j 2^{-k}),\; x_0-\alpha (2^{-k-1} + b_j 2^{-k}) + r]
	\]
	for $j=i,i+1$. Since $x_0-\alpha (2^{-k-1} + b_j 2^{-k}) \in I \ {\rm mod}\ a$, and since $r<2^{-k-3}$, there exists $l\in \mathbb{Z}$ so that
	\[
	[x_0-\alpha (2^{-k-1} + b_j 2^{-k}),\; x_0-\alpha (2^{-k-1} + b_j 2^{-k}) + r] \subset [l 2^{-k} + 3\cdot 2^{-k-3},\; l2^{-k}+5\cdot 2^{-k-3}].
	\]
	In particular, the vertical sides of the parallelograms $R_i$ are contained in the middle half of a slit $s_{l,j}^k$. Therefore, the Lemma \ref{lem:parallelogram} together with the fact that the non-vertical sides of each $R_i$ have slope $\alpha$ gives that the path distance $d_k$ in $[0,1]^2 \setminus \bigcup_{l,j} s_{l,j}^k$ between the interiors of distinct $R_i$ and $R_j$ satisfies
	\[
	d_k(R_i,R_j)\geq 2^{-k-3}>r \quad\text{for all } i\neq j.
	\]
	By the construction of the metric on $\mathbb{M}_\infty$ we obtain $d(\pi_0^{-1}(R_i), \pi_0^{-1}(R_j))>r$. Consequently, each $r$-component of $f^{-1}(W)=\pi_0^{-1}(\tilde{W})$ must lie within a single $\pi_0^{-1}(R_j)$.
	
	Next, $\pi_0$ is $L$-Lipschitz regular by Lemma \ref{lem:pi-regular}. Thus, each $\pi_0^{-1}(R_j)$ is contained within at most $L$ balls $\{B_1,\dots,B_\ell\}$ with diameter at most $17(C_1+1)r$ (where $\ell\le L$). Hence each $r$-component of $\pi_0^{-1}(R_j)$ is contained within the union of at most $L$ such balls s.t. the incidence graph of these balls inflated by two is connected. Therefore the diameter of the union, which contains the component, is at most $100(C_1+1)Lr$. This gives Lipschitz lightness with a constant $C$ satisfying, for example,
	\[
	C \ge 100 (C_1+1) L.
	\]
\end{proof}

With the previous work, Theorem \ref{thm:embeddingL1} follows.
\begin{proof}[Proof of Theorem \ref{thm:embeddingL1}]
	Recall that $\mathbb{M}_\infty$ is a length space, and that $f=\Lambda_{\alpha} \circ \pi_0$ is Lipschitz light by Theorem \ref{thm:liplight}. Therefore, the existence of a bi-Lipschitz embedding follows from \cite[Theorem 1.1]{cheegerkleiner}.
\end{proof}

\begin{remark}
	In fact, the theorems in \cite{cheegerkleiner} show more than just the  bi-Lipschitz embeddability. They show that there is an ``admissible inverse limit system of spaces'' \[
	X_0 \shortleftarrow X_1 \shortleftarrow X_2 \shortleftarrow \cdots \shortleftarrow X_n \shortleftarrow \cdots\]
	so that $\mathbb{M}_\infty$ is bi-Lipschitz to the inverse limit of the spaces $X_n$. The authors find this fact somewhat curious, and slightly counterintuitive. However, since the construction in \cite{cheegerkleiner} is explicit, and since the map $\Lambda_{\alpha} \circ \pi_0$ is also explicit, it should be possible to compute the inverse limit system explicitly. There may be some interest to pursue this question, in order to get a better geometric grasp of the embedding constructed.
\end{remark}
	
	\section{The $\beta$--number approach for dyadic slit carpets}\label{betaNumberChap}
	
	\noindent In this section we study the family of dyadic slit carpets $\mathbb{M}^{a}$. These metric spaces are obtained from the unit square by removing vertical slits whose lengths decrease relative to the dyadic scale; our aim is to relate the geometric size of the slits to (non-)embeddability of $\mathbb{M}^a$ into Euclidean spaces via bi-Lipschitz maps. Compare with the non-embeddability results for slit-type carpets; see \cite{seb2020regular}. The main tool is the theory of Jones' $\beta$--numbers and the corresponding travelling--salesman type estimates.
	
	Let $a=(a_n)_{n=1}^\infty$ be a non-increasing sequence with $0<a_n<1$ for every $n$. For each dyadic square $Q$ of side length $2^{-n}$ remove from its center a vertical slit of length $a_n2^{-n-1}$. Equip the remainder after all of these removals with the path metric, and denote by $\mathbb{M}^a\subset [0,1]^2$ the resulting compact metric space obtained from completing this space. There exists a $1$-Lipschitz projection $\pi:\mathbb{M}^a\to [0,1]^2$. For each of the slits $s$, the points $p\in s$ in its interior get doubled in $\mathbb{M}^a$ and $\pi^{-1}(p)$ consists of two points. For other points, $\pi$ is injective, and we will routinely identify the point with its image in this section.
	
	When $a_n$ is the constant sequence $(1,1,...)$, this construction reduces to the standard slit carpet $\mathbb{M}_\infty$, and when $a_n \to 0$, we obtain a space that infinitesimally resembles Euclidean space. A natural question arises: If $(a_n)_{n\geq 1}$ decays slowly, does $\mathbb{M}^a$ embed into $\ell_2$? We will prove a necessary condition for embeddability: $(a_n)_{n\geq 1} \in \ell_{1+\epsilon}$ for $\epsilon>0$. This gives also another proof that $\mathbb{M}_\infty$ does not bi-Lipschitz embed to $\ell_2$. 
	
	To quantify the local deviation of a function from being linear we use the usual $\beta$-numbers for mappings. If $X $ is a Hilbert space and $f:I\to X$ is continuous on an interval $I\subset\mathbb{R}$ we define
	\[
	\beta(f,I):=\frac{1}{|I|}\inf_{b\in\mathbb{R}^d,\;c\in\mathbb{R}^d}\sup_{x\in I}|f(x)-b-cx|.
	\]
	Thus $\beta(f,I)$ measures the best uniform deviation of $f$ on $I$ from an affine map, normalized by the length of $I$. We will also use the notation
	\[
	h_I:=\frac{1}{|I|}\int_I h(t)\,dt
	\qquad\text{and}\qquad
	\langle h,g\rangle:=\int_0^1 h(t)g(t)\,dt
	\]
	for the average and the $L^2([0,1])$ inner product respectively. Let $\mathcal{D}_n$ denote the collection of closed dyadic intervals of length $2^{-n}$.
	
	The key quantitative input is the one--dimensional ``function'' version of Jones' travelling--salesman estimate (see \cite{schul2007subsets} for the curve versions). A similar result has also appeared in Cheeger's paper \cite{cheeger2012quantitative} and the proof of the following formulation can be adapted from \cite[Section X.2]{garnett}. It is worth to note the scaling $L^2$, which comes from scaling, and has appeared sometimes incorrectly in the literature. This is in contrast to the usual traveling Salesman theorem where only the length appears on the right-hand side.
	
	\begin{lemma}\label{lem:betasum}
		Let $f:[0,1]\to X$ (where X is $\mathbb{R}$ or any Hilbert space) be $L$-Lipschitz. Then
		\[
		\beta(f):=\sum_{n=0}^\infty\sum_{I\in\mathcal{D}_n}\beta(f,4I)^2\,|I|\leq C L^2,
		\]
		where $C>0$ is an absolute constant.
	\end{lemma}
	
		\noindent Here, $4I$ is the interval with the same center as $I$ and length $4|I|$. Our analysis of $\mathbb{M}^a$ begins by applying this estimate fiberwise to the restrictions of a bi-Lipschitz map $f:\mathbb{M}^a$ to vertical segments in the slit carpet. For a related proof see the splitting result in \cite{dsebsplit}. 

		For $t\in[0,1]$ we denote by $\gamma_t:[0,1]\to\mathbb{M}^a$ the vertical path $\gamma_t(s)=(t,s)$ whenever this path is well defined in $\mathbb{M}^a$. This fails only for a countable set of exceptional $t$ corresponding to the vertical projections of endpoints of removed slits. We shall study $\beta(F\circ\gamma_t,I)$ for dyadic intervals $I$.
	
	The following alternative captures the dichotomy that appears near a slit: either the image of the two sides of the slit are close, or at least one of the adjacent vertical fibers exhibits a definite amount of nonlinearity as measured by the $\beta$-numbers.
	
\begin{lemma}\label{lem:alternative}
	Let $Q=[a,b]\times[c,d]\in\mathcal{D}_n$ be a dyadic square of side length $2^{-n}$, and let $F:\mathbb{M}^a\to\mathbb{R}$ be $L$-Lipschitz. Denote by $m_Q^\pm$ the two midpoints on the opposite sides of the central vertical slit of $Q$. Set
	\[
	I:=\left[\frac{c+d}{2}-\frac{a_n}{2}2^{-n-1},\,\frac{c+d}{2}+\frac{a_n}{2}2^{-n-1}\right].
	\]
	Then, either
	\begin{enumerate}
		\item $|F(m_Q^+)-F(m_Q^-)|\leq a_n2^{-n-2}$, or
		\item for every $t\in[0, L^{-1} a_n2^{-n-5}]$ (with $m=(a+b)/2$) one has
		\[
		\beta(F\circ\gamma_{m-t},I)^2+\beta(F\circ\gamma_{m+t},I)^2\geq 2^{-10}.
		\]
	\end{enumerate}
\end{lemma}

\begin{proof}
	Assume that the second alternative fails, and let $\delta=2^{-5}$. Then for some
	\[
	t\in \left[0,\frac{1}{L}a_n2^{-n-5}\right]
	\]
	we have
	\[
	\beta(F\circ\gamma_{m-t},I)\leq 2^{-5}
	\qquad\text{and}\qquad
	\beta(F\circ\gamma_{m+t},I)\leq 2^{-5}.
	\]
	
	Thus there are constants $c_0^\pm,c_1^\pm$ such that
	\[
	|F(\gamma_{m\pm t}(y))-c_0^\pm-c_1^\pm y|\leq a_n2^{-n-5}
	\]
	for all $y\in I$. Set
	\[
	u:=\frac{c+d}{2}-\frac{a_n}{2}2^{-n-1},
	\qquad
	v:=\frac{c+d}{2}+\frac{a_n}{2}2^{-n-1},
	\qquad
	y_0:=\frac{u+v}{2}=\frac{c+d}{2}.
	\]
	Applying this to $y=u,y_0,v$, we get
	\begin{align}
		|F(\gamma_{m\pm t}(u))-c_0^\pm-c_1^\pm u| &\leq a_n2^{-n-5}, \nonumber \\
		|F(\gamma_{m\pm t}(v))-c_0^\pm-c_1^\pm v| &\leq a_n2^{-n-5}, \nonumber \\
		|F(\gamma_{m\pm t}(y_0))-c_0^\pm-c_1^\pm y_0| &\leq a_n2^{-n-5}. \label{eq:betanumberest}
	\end{align}
	
	Next, consider the construction of the dyadic slit carpet, and observe that no slit intersects the horizontal sides of a dyadic square. Since
	\[
	t\in \left[0,\frac{1}{L}a_n2^{-n-5}\right]
	\]
	and, for example, $d(\gamma_{m+t}(u),\gamma_{m-t}(u))=2t$, using the Lipschitz bound for $F$ we get:
	\begin{align}
		d(\gamma_{m+t}(u),\gamma_{m-t}(u)) &\leq \frac{1}{L}a_n2^{-n-4}
		\Longrightarrow
		|F(\gamma_{m+t}(u))-F(\gamma_{m-t}(u))| \leq a_n2^{-n-4}, \nonumber \\
		d(\gamma_{m+t}(v),\gamma_{m-t}(v)) &\leq \frac{1}{L}a_n2^{-n-4}
		\Longrightarrow
		|F(\gamma_{m+t}(v))-F(\gamma_{m-t}(v))| \leq a_n2^{-n-4}, \nonumber \\
		d(\gamma_{m+t}(y_0),m_Q^+) &\leq \frac{1}{L}a_n2^{-n-5}
		\Longrightarrow
		|F(\gamma_{m+t}(y_0))-F(m_Q^+)| \leq a_n2^{-n-5}, \nonumber \\
		d(\gamma_{m-t}(y_0),m_Q^-) &\leq \frac{1}{L}a_n2^{-n-5}
		\Longrightarrow
		|F(\gamma_{m-t}(y_0))-F(m_Q^-)| \leq a_n2^{-n-5}. \label{eq:lipboundf}
	\end{align}
	
	We still need one more estimate, which is obtained by combining the triangle inequality (TI) with \eqref{eq:betanumberest} and \eqref{eq:lipboundf}, and the fact that $y_0=\frac{1}{2}(u+v)$:
	\begin{align}
		|c_0^+ + c_1^+ y_0 - (c_0^- + c_1^- y_0)| 
		&\overset{\rm TI}{\leq}
		\frac{1}{2}|c_0^+ + c_1^+ u - (c_0^- + c_1^- u)|
		+ \frac{1}{2}|c_0^+ + c_1^+ v - (c_0^- + c_1^- v)| \nonumber \\
		&\overset{\rm TI}{\leq}
		\frac{1}{2}|c_0^+ + c_1^+ u - F(\gamma_{m+t}(u))|
		+ \frac{1}{2}|F(\gamma_{m-t}(u))-(c_0^- + c_1^- u)| \nonumber \\
		&\qquad +\frac{1}{2}|c_0^+ + c_1^+ v - F(\gamma_{m+t}(v))|
		+ \frac{1}{2}|F(\gamma_{m-t}(v))-(c_0^- + c_1^- v)| \nonumber \\
		&\qquad \qquad + \frac{1}{2}|F(\gamma_{m+t}(u))-F(\gamma_{m-t}(u))|
		+ \frac{1}{2}|F(\gamma_{m+t}(v))-F(\gamma_{m-t}(v))| \nonumber \\
		&\overset{\eqref{eq:betanumberest},\eqref{eq:lipboundf}}{\leq} a_n2^{-n-3}. \label{eq:midpointslines}
	\end{align}
	
	Using these, and the triangle inequality, we get the following estimate for how much $F$ stretches the midpoints $m_Q^+,m_Q^-$:
	\begin{align*}
		|F(m_Q^+)-F(m_Q^-)| 
		&\overset{\rm TI}{\leq}
		|F(m_Q^+)-F(\gamma_{m+t}(y_0))|
		+|F(m_Q^-)-F(\gamma_{m-t}(y_0))| \\
		&\qquad +|F(\gamma_{m+t}(y_0))-F(\gamma_{m-t}(y_0))| \\
		&\overset{\eqref{eq:lipboundf}}{\leq}
		a_n2^{-n-4}+|F(\gamma_{m+t}(y_0))-F(\gamma_{m-t}(y_0))| \\
		&\overset{\rm TI}{\leq}
		a_n2^{-n-4}
		+|F(\gamma_{m+t}(y_0))-c_0^+ -c_1^+ y_0| \\
		&\qquad + |F(\gamma_{m-t}(y_0))-c_0^- -c_1^- y_0|
		+ |c_0^+ + c_1^+ y_0 - (c_0^- + c_1^- y_0)| \\
		&\overset{\eqref{eq:betanumberest}}{\leq}
		a_n2^{-n-4}+ a_n2^{-n-4} + a_n2^{-n-3} \\
		&\overset{\eqref{eq:midpointslines}}{\leq}
		a_n2^{-n-2}.
	\end{align*}
    This gives the first alternative, and completes the proof.
\end{proof}

By assuming that we have a $(1,L)$-bi-Lipschitz embedding $F:\mathbb{M}^a\to \ell_2$, we can prevent the first alternative in the previous lemma. Thus, $\beta$-numbers are large at least for some vertical lines and some scales. The next step is to integrate over $t$ and sum over all dyadic intervals. For this integration step, we need a small lemma on the existence of a particular type of measure on $[0,1]$.

    \begin{lemma}\label{lem:measure} Let $\epsilon \in (0,1)$.
        There exist a measure $\mu$ on $[0,1]$ and a constant $C>0$, so that for every dyadic interval $I$ with midpoint $c$ and any $\delta\in (0,1)$ we have
        \begin{equation}\label{eq:lb}
        \mu([c-\delta|I|,c+\delta |I|])\geq C\delta^\epsilon \mu(I).
        \end{equation}
    \end{lemma}
    \begin{proof}
    We will construct the measure $\mu$ as a self-similar measure on $[0,1]$.   Let $m\in \N$ be such that $m\epsilon \geq 2$, and let $\mathcal{I}_0=\{[0,1]\}$. For every $I=[a,b]$, let $\mathcal{I}(I)=\{[a+2^{-m}(b-a)k, a+2^{-m}(b-a)(k+1)] : k=0,\dots, 2^{m}-1\}$ be its subdivision to $2^m$ equal pieces. Let $J_k(I)=[a+2^{-m}(b-a)i, a+2^{-m}(b-a)(k+1)]$, $k=0,\dots, 2^{m}-1$, denote these pieces, and recursively construct $\mathcal{I}_{n+1}=\bigcup_{I\in \mathcal{I}_n} \mathcal{I}(I)$. Choose a parameter $\eta\in (0,1/4)$. With this choice $(\frac{1-\eta}{2})\geq 1/4$. There exists a unique measure $\mu$ on $[0,1]$ that satisfies the following
    \begin{enumerate}
        \item $\mu([0,1])=1$
        \item If $\mu(I)$ has been defined for all $I\in \mathcal{I}_n$, then for $J=J_k(I)$ we have
        \[
        \mu(J)=\begin{cases} \frac{1-\eta}{2}\mu(I) & k=0,2^m-1 \\\frac{\eta}{2^m-2} \mu(I) & k=1,\dots, 2^m-2.\end{cases}
        \]
    \end{enumerate}
    A classic argument using weak convergence shows that such a measure exists, see eg. \cite[Proposition 1.7. and Section 17.3]{falconer}.  It follows from \cite[Proposition 3.3.]{yung}, that $\mu$ is doubling: $\mu(I)\sim \mu(I')$ for any two dyadic intervals with $|I|/|I'|\leq 2$ and $I\cap I'\neq\emptyset$.

    Let $I$ be any doubling interval, and let $l\in \N$ be such that $2^{-m(l-1)}\geq|I|\geq 2^{-ml}$. Choose an interval of length $\widetilde{I}$ that contains $c$ as its left endpoint. By repeatedly applying doubling, we get $\mu(\widetilde{I})\geq L^{-1} \mu(I)$ for some $L\geq 1$ dependent on $m$, but independent of $I$ or $k$.

    By doubling, it is sufficient to prove Equation \eqref{eq:lb} for $\delta=2^{-km}$ and $k\in \N$ and $k\geq 1$. By induction, and the definition of the measure, we get
    \[
    \mu([c-\delta |I|,c+\delta |I|])\geq\mu([c, c+\delta |\widetilde{I}|]) \geq \left(\frac{1-\eta}{2}\right)^k \mu(\widetilde{I})\geq 2^{-2k } \mu(\widetilde{I})\geq 2^{-mk\epsilon} \mu(\widetilde{I}) \geq L^{-1} \delta^\epsilon \mu(I).
    \]
    \end{proof}
	\medskip

	\medskip
We now integrate the estimate on $\beta$ numbers obtained from Lemma \ref{lem:alternative} over $t$ with respect to the measure $\mu$.
	
		\begin{theorem} Let $\epsilon >0$ and assume that $a=(a_i)_{i\in\N}$ is a non-increasing sequence with $a_i\in (0,1)$. If $\mathbb{M}^{a}$ admits a bi-Lipschitz embedding into $\ell_2$, then $a\in \ell_{1+\epsilon}.$ 
	\end{theorem}
	\begin{proof}
    Let $\mu$ be the measure constructed in Lemma \ref{lem:measure}.
		Assume that $f:\mathbb{M}^{a}\to \R^N$ is $(1,L)$-bi-Lipschitz. The proof proceeds by applying Lemma~\ref{lem:alternative} on each dyadic scale and integrating the resulting lower bounds over the relevant horizontal intervals. Since $f$ is $L$-Lipschitz, we get
		\begin{equation}\label{eq:betaintup}
		\int_0^1 \beta(f\circ \gamma_t) d\mu \leq C_1 L^2 
		\end{equation}
        for some constant $C_1>0$.
		
		Next, fix $n\in \N$. For every $k\in \{0,\dots, 2^n-1\} $ consider $t\in [0,a_nL^{-1}2^{-n-5}]$. The function $f$ is $1$-Lipschitz, and (1) in Lemma \ref{lem:alternative} cannot occur for any dyadic square. Thus, for every $l=0,\dots, 2^n-1$, we can apply (2) to the rectangle $[k2^{-n}, (k+1)2^{-n}]\times [2^{-n-1}(2l+1-a_n/2), 2^{-n-1}(2l+1+a_n/2)]$ and obtain
		\begin{align*}
			\beta\left(f\circ \gamma_{(k+\frac{1}{2})2^{-n}-t}, \left[2^{-n-1}(2l+1-a_n/2), 2^{-n-1}(2l+1+a_n/2)\right]\right)^2 + \\
			\beta\left(f\circ \gamma_{(k+\frac{1}{2})2^{-n}+t}, \left[2^{-n-1}(2l+1-a_n/2), 2^{-n-1}(2l+1+a_n/2)\right]\right)^2  &\geq 2^{-10}.
		\end{align*}

        For each $l$ let $I_{l}$ be a dyadic interval with length $a_n 2^{-n-1} < |I_l| \leq a_n2^{-n}$ which intersects $\left[2^{-n-1}(2l+1-a_n/2), 2^{-n-1}(2l+1+a_n/2)\right]$. We have $\left[2^{-n-1}(2l+1-a_n/2), 2^{-n-1}(2l+1+a_n/2)\right] \subset 4I_{k,l}$ and thus $\beta(f\circ \gamma_t, 4I_{l})\geq 8^{-1}\beta(f\circ \gamma_t, \left[2^{-n-1}(2l+1-a_n/2), 2^{-n-1}(2l+1+a_n/2)\right])$. Thus,
		\[
			\beta\left(f\circ \gamma_{(k+\frac{1}{2})2^{-n}-t}, I_{l}\right)^2 + 
			\beta\left(f\circ \gamma_{(k+\frac{1}{2})2^{-n}+t}, I_{l}\right)^2  \geq 2^{-13}.
		\]
		Integrate this over $t\in [0,a_nL^{-1}2^{-n-5}]$ and use \eqref{eq:lb}, to get a constant $C_2>0$ s.t.
		\begin{align*}
		\int_{(k+1/2-a_n2^{-5}L^{-1})2^{-n}}^{(k+1/2+a_n2^{-5}L^{-1})2^{-n}} \beta(f\circ \gamma_{t}, I_{l})^2 dt &\geq 2^{-13} \mu([(k+1/2-a_n2^{-5}L^{-1})2^{-n}, (k+1/2+a_n2^{-5}L^{-1})2^{-n}])\\
        &\geq C_2^{-1} a_n^\epsilon \mu([k2^{-n}, (k+1)2^{-n}]),
		\end{align*}
       
		Sum this over $k=0,\dots, 2^{n}-1$, we get 
		\[
		\int_0^1 \beta(f\circ \gamma_{t}, I_{l})^2 d\mu \geq   C_2^{-1} a_n^\epsilon.
		\]
        Let $m_n$ be such that $|I_l|=2^{-m_n}$. Notice that for $l=0,\dots, 2^{n}-1$, the intervals $I_l$ are subsets of different dyadic intervals at level $n$ and thus distinct. By multiplying this by the length of $|I_l|\geq a_n 2^{-n-1}$, and summing over $l=0, \dots, 2^{n}-1$, we thus get
        \[
		\int_0^1 \sum_{J\in \mathcal{D}_{m_n}} \beta(f\circ \gamma_{t}, J)^2|J| d\mu \geq   2^{-1}C_2^{-1} a_n^{1+\epsilon}.
		\]
		Sum this over $n$ to get
		\[
		\int_0^1 \sum_{n\in\N}\sum_{J\in \mathcal{D}_{m_n}} \beta(f\circ \gamma_{t}, J)^2|J| d\mu \gtrsim \sum_{n=1}^\infty a_n^{1+\epsilon}.
		\]
        Since $(a_n)_{n\in\N}$ is non-increasing, it is easy to see that $(m_n)_{n\in\N}$ is a strictly increasing sequence.		By combining this with \eqref{eq:betaintup}, we obtain the fact that $a\in \ell_{1+\epsilon}$.
	\end{proof}

If $\mu$ from Lemma \ref{lem:measure} is replaced with Lebesgue measure $\lambda$, one obtains just the condition $a\in \ell_2$. The construction of the strange measure $\mu$ is needed for this stronger statement. 
	
It seems plausible that if $a\in \ell_1$, that $\mathbb{M}^{a}$ embeds into $\ell_2$. We leave this question, and the question about sharpness of the above theorem for future study. It seems interesting to also explore a similar problem for the Laakso diamond graph, where the relative sizes of the diamonds is decreased. These would be obtained by varying the construction in \cite{lang2001bilipschitz}.

	\section{Qualitative nonembeddability of general slit carpets}\label{genSlitCarpetSect}
	\noindent We now return to the qualitative embeddability results and generalize the main result of \cite{seb2020regular} in two directions: we consider more general carpets, where slits need not be placed in the center of squares, and more general RNP Banach space targets. 

A slit is a set of the form
\[ 
S := \{x_S\} \times (a_S,b_S), \qquad 0<a_S<b_S<1,\]
 and we set $\ell(S):=b_S-a_S$. Consider a countable collection of non-overlapping slits $\mathcal{S}$ in the unit square $[0,1]^2$. Consider the completion $\mathbf{Z}$ of the set
\[
Z = [0,1]^2 \setminus \bigcup_{S\in \mathcal{S}} S,
\]
equipped with the path metric $d_Z$. (Technically, this is well-defined only if $d_Z$ is finite for every pair of points --- which will follow for generalized slit carpets).  We call $\mathbf{Z}$
 a \textit{general vertical slit carpet} if there exist constants $C_0,\delta>0$ such that:
\begin{enumerate}
	\item 	(Slits at every scale and location.) For every $z\in[0,1]^2$ and every $r\in(0,1)$ there is a slit $S \subset B(z,r)$  and $\delta r \le \ell(S_\alpha) \le C_0 r.$ 
    \item (Slits decay in size) For every $\epsilon>0$ $\{S\in \mathcal{S} : \ell(S)>\epsilon\}$ is finite.
	\item (Horizontal corridor (h–condition).)
	There exists a constant $\eta \in (0,1)$ such that for every slit $S\in\mathcal{S}$
	the two horizontal segments
	\[
	[x_\alpha - \eta \ell_\alpha,\ x_\alpha + \eta \ell_\alpha] \times \{a_\alpha\}
	\quad\text{and}\quad
	[x_\alpha - \eta \ell_\alpha,\ x_\alpha + \eta \ell_\alpha] \times \{b_\alpha\}
	\]
	are contained in $Z$.
	
\end{enumerate}

In particular, the $h$-condition implies that for any
	\[
	x_1, x_2 \in [x_\alpha - \eta \ell_\alpha,\ x_\alpha + \eta \ell_\alpha],
	\]
	the horizontal segments
	\[
	[x_1,x_2] \times \{a_\alpha\}
	\quad\text{and}\quad
	[x_1,x_2] \times \{b_\alpha\}
	\]
	lie entirely in $S$, and
	\[
	d_Z\big((x_1,a_\alpha),(x_2,a_\alpha)\big) = |x_1 - x_2|,
	\qquad
	d_Z\big((x_1,b_\alpha),(x_2,b_\alpha)\big) = |x_1 - x_2|.
	\]
    Since no slit intersects the top or the bottom of the square $[0,1]^2$, it is easy to see by using the $h$-condition that the distance in $S$ between pairs of points is always finite, and thus $\mathbf{Z}$ is a true metric space. Most of our analysis below will be done in the subset $Z$ contained isometrically in $\mathbf{Z}$.

	We are considering bi-Lipschitz mappings from the general slit carpet domains to RNP targets, which are defined as follows. See \cite[Chapter 2]{pisier} for further background.
	\begin{definition}\label{def:rnp}
		A Banach space $X$ satisfies the RNP property if every Lipschitz map $\gamma:[0,1]\to X$ is differentiable almost everywhere.
	\end{definition}
    We will also use the fact that the derivative $\gamma'(t)$ thus obtained is Bochner measurable and integrable; see \cite{diestel1977vector} for some details on these notions.

	We will use the notion of \emph{nicely shrinking sets} from \cite[Subsection 7.9]{walter1987real}.
	
	\begin{definition}\label{def:niceShrink}
		Given a point $x\in\mathbb{R}^n$, we say that a sequence of sets $\{E_i\}$ shrinks nicely to $x$ if there exists a constant $c>0$ and a sequence of balls $\{B(x,r_i)\}$ with $r_i\to 0$ such that
		\[
		\lambda(E_i)\geq c\,\lambda(B(x,r_i))\qquad\text{and}\qquad E_i\subset B(x,r_i)
		\]
		for every $i\in \mathbb{N}$.
	\end{definition}

    Before we prove Theorem \ref{thm:RNPnonemb}, we introduce some notation and auxiliary tools. 	Fix $\varepsilon\in(0,\eta/2)$. For each slit $S\in \mathcal{S}$ define the left and right vertical neighborhoods
	\[
	S_{-}^\varepsilon
	:= [x_S-\varepsilon \ell_S,\ x_S]\times[a_S,b_S],
	\qquad
	S_{+}^\varepsilon
	:= [x_S,\ x_S+\varepsilon \ell_S]\times[a_S,b_S].
	\]
	
	By the ``slits at every scale and location'' condition in the definition of a general vertical slit carpet, for every point $p\in Z$ there exists a sequence of slits $S_{n} \in \mathcal{S}$ with $\delta 2^{-n}\leq \ell(S_n)\leq C_0 2^{-n}$ such that
	\[
	S_n\subset B\big(p,2^{-n}\big), n\in \N.
	\]
	Along a suitable subsequence (still indexed by $n$), the families
	\[
	\{S_{n,-}^\varepsilon\}_n,
	\qquad
	\{S_{n,+}^\varepsilon\}_n
	\]
	shrink nicely to $p$ in the sense of Definition~\ref{def:niceShrink}. 
    
	Our main tool in proving Theorem \ref{thm:RNPnonemb} is the use of the vertical ($y$-) derivative to obtain rigidity for Lipschitz embeddings $f$. Let $f:Z\to X$ be a $d_Z$-Lipschitz map. The conditions on the generalized slit carpet imply that $f$ is continuous at a.e. point of $Z$ with respect to the Euclidean topology restricted to $Z$, and thus $f$ is measurable with respect to the Lebesgue measure $\lambda$ on $[0,1]^2$.     For each $x\in[0,1]$, consider the vertical parametrization
	\[
	\gamma_x(y) = (x,y),\qquad y\in[0,1].
	\]
	For almost every $x\in[0,1]$, the vertical fiber $\{x\}\times[0,1]$ meets no slit, so $\gamma_x$ defines a unique segment in $Z$. For such $x$, we define
	\[
	\partial_y f(x,y) = (f\circ\gamma_x)'(y)
	\]
	on those $y$ for which the (metric) derivative exists; this holds for almost every $y\in[0,1]$ by the RNP property of $X$. The resulting function $\partial_y f$ is Bochner measurable, see e.g.\ \cite{diestel1977vector}. On the remaining null set we set $\partial_y f=0$.
	
	The guiding intuition is that the $y$-derivative must become asymptotically constant. Consequently, the curves $\gamma_x$ are mapped to nearly parallel line segments. This, in turn, forces points on opposite sides of a slit to be mapped close to one another, leading to a contradiction. We make this precise using the following lemma.
	
	For each slit $S\in \mathcal{S}$ set
	\[
	y_S:=\frac{a_S+b_S}{2},
	\qquad
	m_{-}^{S,t} := \big(x_S-\tfrac{t}{2}\ell_S,\,y_S\big),\quad
	m_{+}^{S,t} := \big(x_S+\tfrac{t}{2}\ell_S,\,y_S\big).
	\] 
    Let $I_S$ be the set of $t$ s.t. $m_\pm^t\in C$ and $\gamma_{x_S-\tfrac{t}{2}\ell_S} \subset Z$.
	
	\begin{lemma}\label{lem:midpoints} Assume that $X$ is an RNP Banach space.
		Let $f:Z\to X$ be $1$-Lipschitz with respect to $d_Z$. Fix $\varepsilon\in (0,\eta/2)$. Then for every slit $S\in \mathcal{S}$ and every $t\in I_S$
		\[
		  \|f(m_-^{S,t})-f(m_+^{S,t})\|
		\leq\left\|
		\int_{a_S}^{y_S} \partial_y f(x_S-t, s)\, d\lambda_s
		-
		\int_{a_S}^{y_S} \partial_y f(x_S+t,s) d\lambda_s
		\right\|
		+ 2\,\varepsilon\, \ell_{S}.
		\]
	\end{lemma}
	
	\begin{proof} Let $t\in I_S$ has full measure for every $S\in \mathcal{S}$. 
		First, it is sufficient to assume that $X=\R$, since by the Hahn--Banach theorem we may choose $\lambda\in X^*$ with $\|\lambda\|_{X^*}=1$ and
		\[
		\lambda(f(m_-^{S,t})-f(m_+^{S,t}))=\|f(m_-^{S,t})-f(m_+^{S,t})\|_X.
		\]
		With these choices, we have
		\[
		|\lambda(f(m_-^{S,t})-f(m_+^{S,t}))|
		=\|f(m_-^{S,t})-f(m_+^{S,t})\|_X,
		\]
        and $\lambda \circ \partial_y f = \partial_y (\lambda \circ f)$. Thus proving the lemma for the $L$-Lipschitz function $g=\lambda\circ f$ in place of $f$ yields the claim for $f$.
		
		Fix a slit $S$ and write $a_S,b_S,x_S,\ell_S$ as above. Consider the two vertical curves
		\[
		\gamma^t_{\pm,S}(s)=(x_S\pm t,s),\qquad s\in [a_S,b_S].
		\]
		   By the horizontal corridor condition at the endpoint $a_S$, for all $t\in[0,\varepsilon\ell_S]$ the horizontal segment
		\[
		[x_S-t,\,x_S+t]\times\{a_S\}
		\]
		lies in $Z$, and hence
		\[
		d_Z\big(\gamma_{+,S}^t(a_S),\gamma_{-,S}^t(a_S)\big)
		=2t\le 2\varepsilon\ell_S.
		\]
		Moreover, by the definition of $m_\pm^{S,t}=\gamma^t_{+,S}(y_S)$ for a.e. $t\in I_S$.
		
		Thus,
		\begin{equation}\label{eq:bottomdistance-general}
			\|f(\gamma_{+,S}^t(a_S))-f(\gamma_{-,S}^t(a_S))\|
			\leq 2\varepsilon \ell_S,
		\end{equation}
		since $f$ is $1$-Lipschitz.
		
		Next, we can compute
		\begin{equation}\label{eq:midpoint-general}
			f(\gamma_{\pm,S}^t(y_S))-f(\gamma_{\pm,S}^t(a_S))
			= \int_{a_S}^{y_S} \partial_y f(\gamma_{\pm,S}^t(s))\,ds.
		\end{equation}
		Summing the previous three inequalities, we obtain for almost every $t\in[0,\varepsilon \ell_S]$ that
		\begin{align*}
			f(m_+^{S,t})-f(m_-^{S,t})
			\;\leq\; &\big(f(\gamma_{+,S}^t(y_S))-f(\gamma_{+,S}^t(a_S))\big)
			- \big(f(\gamma_{-,S}^t(y_S))-f(\gamma_{-,S}^t(a_S))\big) \\
			&+ \big|f(\gamma_{+,S}^t(a_S))-f(\gamma_{-,S}^t(a_S))\big| \\
			\leq\; &\int_{a_S}^{y_S} \partial_y f(\gamma_{+,S}^t(s))\,ds
			- \int_{a_S}^{y_S} \partial_y f(\gamma_{-,S}^t(s))\,ds
			+ 2\epsilon\,\varepsilon \ell_S,
		\end{align*}
		which yields the claim.
	\end{proof}

We are now ready to prove Theorem \ref{thm:RNPnonemb} that states that a generalized slit carpet does not bi-Lipschitz embed into any Banach space with the RNP property. The argument will be based on the $y$-derivative $\partial_y f$, and its measurability. By zooming in, and Lebesgue differentiation, this derivative becomes  nearly constant. However, this causes a problem, since slits occur at every scale and location and Lemma \ref{lem:midpoints} forces there to be a difference in derivatives at the left- and right-hand sides of slits.

	\begin{proof}[Proof of Theorem \ref{thm:RNPnonemb}]
		Let $X$ be any Banach space with the RNP property.
		For the sake of contradiction, let $f:Z\to X$ be a $(\delta,1)$-bi-Lipschitz map for some $\delta\in (0,1)$. Such an embedding is obtained by restricting and scaling a bi-Lipschitz embedding of $\mathbf{Z}$.
		
		Let $p \in Z$ be a Lebesgue point of $\partial_y f$. Since $f$ is Lipschitz, the partial derivative $\partial_y f$ is bounded and thus Bochner integrable. By \cite[Theorem II.2.9]{diestel1977vector} and \cite{Bochner1933}, we can apply the Lebesgue differentiation theorem to the vector-valued Bochner integrable function $\partial_y f$ along nicely shrinking sets. In particular, using the sequences $\{S_{n,-}^\epsilon\}_n$ and $\{S_{n,-}^\epsilon\}_n$ shrinking nicely to $p$, we obtain for every $\epsilon>0$
		\begin{align*}
			\lim_{n\to\infty}
			\vint_{S_{n,-}^\epsilon} \|\partial_y f - \partial_y f(p)\|\,d\mu
			+
			\vint_{S_{n,-}^\epsilon}\|\partial_y f- \partial_y f(p)\|\,d\mu
			=0.
		\end{align*}
		
		Let $\epsilon = \delta/4$  and choose $n$ large enough so that
		\begin{equation}\label{eq:intupperbound}
		\vint_{S_{n,-}^\epsilon} \|\partial_y f - \partial_y f(p)\|\,d\mu
			+
			\vint_{S_{n,-}^\epsilon}\|\partial_y f- \partial_y f(p)\|\,d\mu
		\leq \delta/4.
		\end{equation}
		By Lemma \ref{lem:midpoints} applied to the slit $S_{n}$ and $t\in I_{S_n}$ we get
		\begin{align*}
		\|f(m_+^{S_n,t})-f(m_-^{S_n,t})\| &\leq \left\|
		\int_{a_{S_n}}^{y_{S_n}} \partial_y f(x_{S_n}-\tfrac{t}{2}\ell_{S_n}, s)\, d\lambda_s
		-
		\int_{a_{S_n}}^{y_{S_n}} \partial_y f(x_{S_n}+\tfrac{t}{2}\ell_{S_n}\,s) d\lambda_s
		\right\|
		+ 2\,\varepsilon\, \ell_{S_n} \\
        &\leq 
		\int_{a_{S_n}}^{y_{S_n}} \|\partial_y f(x_{S_n}-\tfrac{t}{2}\ell_{S_n}, s)-\partial_y f(p)\|\, d\lambda_s
		+
		\int_{a_{S_n}}^{y_{S_n}} \|\partial_y f(x_{S_n}+\tfrac{t}{2}\ell_{S_n}\,s)-\partial_y f(p)\| d\lambda_s \\
		& \qquad \qquad + 2\,\varepsilon\, \ell_{S_n}. \\
		\end{align*}
        
		On the other hand, $d_Z(m_+^{S_n,t},m_-^{S_n,t})\geq \ell_{S_n}$ since any path  must go either over or under the slit. Combining this with an integration over $t\in I_S$ (which has full measure in $[0,\epsilon \ell_S]$) yields
        \begin{align*}
		\epsilon \delta \ell_{S_n}^2 &\leq 
		\int_{S_{n,-}^\epsilon} \|\partial_y f-\partial_y f(p)\|\, d\lambda
		+
		\int_{S_{n,+}^\epsilon} \|\partial_y f-\partial_y f(p)\| d\lambda \\
		& \qquad \qquad + 2\,\varepsilon^2\, \ell_{S_n}^2. \\
		\end{align*}
        Divide this by $\epsilon \ell_{S_n}^2$, and use that $\epsilon = \delta/4$ to get
        \[
        \delta \ell_{S_n}^2/2 \leq \vint_{S_{n,-}^\epsilon} \|\partial_y f-\partial_y f(p)\|\, d\lambda
		+
		\vint_{S_{n,+}^\epsilon} \|\partial_y f-\partial_y f(p)\| d\lambda,
        \]
        which is a contradiction to \eqref{eq:intupperbound}. Thus no such bi-Lipschitz embedding $f$ can exist.
	\end{proof}

		\bibliographystyle{alpha}
\bibliography{bibliography}

\end{document}